\documentclass[a4paper,10pt]{amsart}
\usepackage[utf8]{inputenc}
\usepackage{amssymb}
\usepackage{amsmath}
\usepackage{bbm}
\usepackage{mathrsfs}
\usepackage[all]{xy}
\usepackage{manfnt}
\usepackage{pigpen}
\usepackage{bm}
\usepackage{graphicx}
\usepackage{enumitem}
\usepackage[left=2cm,right=2cm,top=2cm,bottom=2cm]{geometry}
\usepackage{tikz}
\title{Outer functors and a general operadic framework}
\date{}
\author{Geoffrey Powell}
\address{Univ Angers, CNRS, LAREMA, SFR MATHSTIC, F-49000 Angers, France}
\email{Geoffrey.Powell@math.cnrs.fr}
\urladdr{https://math.univ-angers.fr/~powell/}
\keywords{Functors on free groups; polynomial functor; analytic functor; outer functor; Lie operad; PROP; Lie algebra homology}
\subjclass[2000]{18A25, 18M70, 18M85, 13D03, 17B01}
\newtheorem{THM}{Theorem}
\newtheorem{PROP}[THM]{Proposition}

\newtheorem{thm}{Theorem}[section]
\newtheorem{prop}[thm]{Proposition}

\newtheorem{cor}[thm]{Corollary}
\newtheorem{lem}[thm]{Lemma}

\theoremstyle{definition}
\newtheorem{defn}[thm]{Definition}
\newtheorem{exam}[thm]{Example}

\theoremstyle{remark}
\newtheorem{rem}[thm]{Remark}
\newtheorem{nota}[thm]{Notation}
\newtheorem{hyp}[thm]{Hypothesis}

\newcommand{\nt}{\mathrm{Nat}_{V \in \fvs}}
\newcommand{\fvs}{{\mathcal{V}_\rat^{\mathrm{f}}}}
\renewcommand{\phi}{\varphi}
\renewcommand{\hom}{\mathrm{Hom}}
\newcommand{\sym}{\mathfrak{S}}
\newcommand{\gr}{\mathbf{gr}}
\newcommand{\kmod}{\mathtt{Mod}_\kring}
\newcommand{\f}{\mathcal{F}}
\newcommand{\fcatk}[1][\calc]{\f (#1; \kring)}
\newcommand{\fcatQ}[1][\calc]{\f (#1; \rat)}
\newcommand{\nat}{\mathbb{N}}
\newcommand{\ab}{\mathbf{ab}}
\newcommand{\zed}{\mathbb{Z}}
\newcommand{\op}{^\mathrm{op}}
\newcommand{\ob}{\mathrm{Ob}\hspace{2pt}}
\newcommand{\kring}{\mathbbm{k}}

\newcommand{\fb}{{\bm{\Sigma}}}
\newcommand{\opd}{\mathscr{O}}
\newcommand{\fopd}{\f_\opd}

\newcommand{\lie}{\mathfrak{Lie}}
\newcommand{\leib}{\mathfrak{Leib}}
\newcommand{\uass}{\mathfrak{Ass}^u}
\newcommand{\conv}{\odot}
\newcommand{\rconv}{\stackrel{r}{\odot}}
\newcommand{\smodug}{{}_{\kring\fb}\mathtt{Mod}}
\newcommand{\smodopug}{\mathtt{Mod}_{\kring\fb}}
\newcommand{\cat}{\mathbf{Cat}\hspace{1pt}}
\newcommand{\flie}{\f_{\lie}}
\newcommand{\fleib}{\f_{\leib}}
\newcommand{\fppd}{\f_{\mathscr{P}}}

\newcommand{\g}{\mathfrak{g}}
\newcommand{\ppd}{\mathscr{P}}
\newcommand{\foutan}{\f^{\mathrm{Out}}_\omega (\gr\op; \kring)}
\newcommand{\foutanQ}{\f^{\mathrm{Out}}_\omega (\gr\op; \rat)}
\newcommand{\fout}{\f^{\mathrm{Out}}(\gr\op; \kring)}
\newcommand{\foutQ}{\f^{\mathrm{Out}}(\gr\op; \rat)}
\newcommand{\tprop}{\boxplus}
\newcommand{\id}{\mathrm{Id}}
\newcommand{\gropcatuass}{{}_{\Delta} \cat \uass}

\newcommand{\tbar}{\overline{\tau}}
\newcommand{\rbar}{\overline{\rho}}

\newcommand{\mut}{\widetilde{\mu}}
\newcommand{\fgrp}{\mathsf{G}}
\newcommand{\rmod}{\mathtt{Mod}_\lie}
\newcommand{\rmodO}{\mathtt{Mod}_{\cat\opd}}
\newcommand{\rat}{\mathbb{Q}}
\newcommand{\bimodO}{{}_{\cat\opd} \mathtt{Mod}_{\cat\opd}}
\newcommand{\ad}{\mathrm{ad}}

\numberwithin{equation}{section}
\begin{document}

\begin{abstract}
For $\opd$ an operad in $\kring$-vector spaces,  $\fopd$  is defined to be the category of $\kring$-linear functors from the PROP associated to $\opd$ to $\kring$-vector spaces. Given a binary operation $\mu \in \opd (2)$ that satisfies a right Leibniz condition, the full subcategory $\fopd^\mu \subset \fopd$ is introduced here and its properties studied.

This is motivated by the case of the Lie operad $\lie$, where $\mu$ is taken to be the generator. By previous results of the author,  when $\kring = \rat$, $\flie$ is equivalent to the category of analytic functors on the opposite $\gr\op$ of the category $\gr$ of finitely-generated free groups. The main result shows that $\flie^\mu$ identifies with the category of {\em outer} analytic functors on $\gr\op$, as introduced in earlier work of the author with Vespa.

Using this identification, this theory has applications to the study of the higher Hochschild homology functors related to work of Turchin and Willwacher.
\end{abstract}

\maketitle

\section{Introduction}
\label{sect:intro}

This paper is motivated by the study of the category of {\em outer functors} on the category $\gr$ of finitely-generated free groups, that was introduced in the joint work with Christine Vespa \cite{PV}. Here we focus upon the category $\fcatk[\gr\op]$ of contravariant functors on $\gr$ with values in $\kring$-vector spaces. The category $\fout$ of outer functors is the full subcategory of  functors $F \in \ob \fcatk[\gr\op]$ for which, for all $n \in \nat$, the inner automorphism group $\mathrm{Inn} (\zed^{\star n}) \subset \mathrm{Aut} (\zed^{\star n}) $ acts trivially on $F(\zed^{\star n})$,  where $\zed^{\star n}$ is the free group on $n$ generators.

The motivating work, \cite{PV}, shows how outer functors appear naturally in connection with the study of the higher Hochschild homology of a wedge of circles, building upon ideas of Turchin and Willwacher \cite{MR3982870}. Subsequent work of Gadish and Hainaut \cite{2022arXiv220212494G} has considered related structures.

Outer functors also arise naturally in other contexts, for example in the work of  Katada \cite{MR4613613,2021arXiv210509072K}. This involves functors on $\gr\op$ constructed from Jacobi diagrams, based on constructions of Habiro and Massuyeau \cite{MR4321214} in their work on generalizations of the Kontsevich integral. These are polynomial functors and Katada's results show that they are outer functors; the framework presented here gives an alternative, natural approach to their study.  More generally, Vespa \cite{V_jac} has considered  functors associated to beaded Jacobi diagrams, thus investigating deeper structure appearing in \cite{MR4321214}. Using the results of this paper, she has analysed the cases for which these  are outer functors, recovering Katada's result as a special case.

There is a natural notion of polynomial functor for functors on $\gr\op$ that is defined using the symmetric monoidal structure on $\gr\op$ given by the free product of groups. Using this, one has the full subcategory of analytic functors $\f_\omega (\gr\op; \kring) \subset \fcatk[\gr\op]$, where a functor is analytic if it is the colimit of its polynomial subfunctors. Taking $\kring=\rat$, there is a `linear algebra' description of $\f_\omega (\gr\op; \rat)$ as the category $\flie$ of representations  of the category $\cat \lie$ associated to the Lie operad $\lie$ (see \cite{P_analytic}). This makes the study of $\f_\omega (\gr\op; \rat)$ much more accessible; for instance, the polynomial filtration can be read off directly when working with $\flie$.

The main purpose of this paper is to identify the full subcategory of $\flie$ that corresponds to the full subcategory of analytic outer functors $\foutanQ \subset \f_\omega (\gr\op; \rat)$. This allows analytic outer functors to be studied entirely within this framework.

It is useful to place this in  the following more  general context. For $\opd$ an operad in $\kring$-vector spaces, one considers the category $\fopd$ of representations of $\cat \opd$, where $\cat \opd$ is the $\kring$-linear category that is associated to $\opd$. The set of objects of $\cat\opd$ is $\nat$ and the category  comes equipped with a  symmetric monoidal structure $\tprop$ which corresponds to addition on objects, so that $n \tprop 1 = n+1$. In particular, this yields the  shift functor $\delta : \fopd \rightarrow \fopd$ that satisfies 
$ 
\delta F (n) = F (n+1) 
$ 
for $F \in \ob \fopd$.

For a binary operation $\mu \in \opd (2)$, for each $n \in \nat$,  a  natural $\kring$-linear map 
\[
\mut (n) : \delta F(n) 
\rightarrow 
F(n)
\]
is defined in Definition \ref{defn:mut}. For example, when $\opd$ is the Lie operad over $\kring = \rat$, by considering universal examples, $\mut (n)$ is  closely related to the right action of the free Lie algebra $\lie (V)$ on $\lie(V)^{\otimes n}$ given by the iterated tensor product of the right adjoint representation, considered naturally with respect to the $\rat$-vector space $V$ (cf. Theorem \ref{THM:proj_cover_mu} below).

In general, $\mut$  does not define a morphism of representations of $\cat\opd$. The following gives a necessary and sufficient condition for this to hold:

\begin{THM}
[Theorem \ref{thm:right_Leibniz}]
The morphisms $\mut(n)$ induce a natural transformation $\mut : \delta \rightarrow \mathrm{Id}$ if and only if $\mu$ satisfies the right Leibniz condition with respect to $\opd$.
\end{THM}

When $\mu$ satisfies this  right Leibniz condition (for example, if $\opd = \lie$ and $\mu \in \lie (2)$ is the generator), then one defines 
\[
\fopd^\mu \subset \fopd
\]
to be the full subcategory of functors $F \in \ob \fopd$ for which  
$
\mut_F : \delta F \rightarrow F 
$ 
is zero.

This subcategory is highly non-trivial. For example, one has the following special case of  Proposition  \ref{prop:alpha_^mu}:

\begin{PROP}
Suppose that the operad $\opd$ is reduced, $\opd (1)= \kring$ and that $\mu$ satisfies the right Leibniz condition. Then   the simple objects of $\fopd$ belong to $\fopd^\mu$.
\end{PROP}

General properties are considered in Sections \ref{sect:mu} and \ref{sect:mu_properties}. For instance, 
combining Propositions \ref{prop:closure} and  \ref{prop:mu_left_adjoint} yields the following starting point for studying the homological relationship between $\fopd$ and $\fopd^\mu$:

\begin{THM}
\label{THM:fopd_mu_properties}
Suppose that $\mu \in \opd (2)$ satisfies the right Leibniz condition. 
\begin{enumerate}
\item 
The subcategory $\fopd^\mu$ is closed under the formation of subobjects, quotients and direct sums.
\item 
The inclusion $\fopd^\mu \subset \fopd$ admits a left adjoint
 $F \mapsto F^\mu$, where $F^\mu := \mathrm{coker} \ \mut_F$, hence:
 \begin{enumerate}
 \item 
$(-)^\mu$ is right exact;
\item 
$\fopd^\mu$ has enough projectives; 
\item 
$(-)^\mu$ preserves projectives. 
\end{enumerate} 
\end{enumerate}
\end{THM}

Let us return to the motivating question, working over $\kring = \rat$ and taking $\opd = \lie$ and $\mu \in \lie(2)$ the generator, so that $\flie^\mu$ is defined. The main result is the following:

\begin{THM}
\label{THM:foutanQ}
[Theorem \ref{thm:outer_mu}]
The category  $\foutanQ \subset \f_\omega (\gr\op; \rat)$ of analytic outer functors is equivalent to the full subcategory $\flie^\mu \subset \flie$.  
\end{THM}

This result makes the study of analytic outer functors on $\gr\op$ more accessible. In particular, via the study of $\flie^\mu$, new information can be obtained upon the structure of the projective generators of $\foutanQ$, as explained in Section \ref{subsect:proj_flie_mu}. 
 This is carried out by using the associated Schur functors (i.e., working with functors on the category $\fvs$ of finite-dimensional $\rat$-vector spaces). There is a left $\cat \lie$-module given by 
\[
\underline{\lie (V)} \ : \  n \mapsto \lie (V)^{\otimes n}
\]
that is natural with respect to the vector space $V$. Here, $\lie(V)$ is the free Lie algebra on $V$ and the action of $\cat\lie$ on $\underline{\lie(V)}$ is induced by the Lie algebra structure of $\lie (V)$.
Now,  using tensor products of  the right adjoint representation of $\lie (V)$, for each $n$,  $\lie (V)^{\otimes n}$ is a right $\lie(V)$-module, naturally with respect to $V$. Passing to zeroth Lie algebra homology (i.e., coinvariants), one obtains a left $\cat \lie$-module 
\[
H_0 (\lie (V); \underline{\lie(V)} ) \ : \  n \mapsto H_0 (\lie(V); \lie(V)^{\otimes n})
\]
that is natural with respect to $V$.

\begin{THM}
\label{THM:proj_cover_mu}
[Theorem \ref{thm:catliemu}]
For $m \in \nat$, the Schur functor associated to the projective cover of $\rat \sym_m$ in $\flie^\mu$ is the homogeneous polynomial component of 
 $
V \mapsto H_0 (\lie(V); \underline{\lie (V)} )$ of degree $m$.
\end{THM}

This is exploited  in \cite{P_wall}; the strategy developed there also suggests that the categories $\fopd^\mu$ for other suitable operads $\opd$ are of interest. In particular, there is a universal example, given by taking $\opd$ to be the (right) Leibniz operad $\leib$. This is related to the Lie operad by the morphism $\leib \rightarrow \lie$ that encodes forgetting antisymmetry; restricting along this gives the diagram of embeddings 
\[
\xymatrix{
\flie^\mu
\ar[r]
\ar[d]
&
\fleib^\mu 
\ar[d]
\\
\flie 
\ar[r]
&
\fleib.
}
\]
This  {\em determines} $\flie^\mu$ in terms of $\fleib^\mu$ (see Example \ref{exam:lie_leib_mu}).

%%%%%%%%%%%%%%%%%%%%%%%%%%%%%%%%%%%%%%%%%%%%%%%%%%%%%%
\subsection{Notation}

The following notation is used throughout:
 \begin{itemize}[label=]
 \item 
 $\nat$ the set of non-negative integers;
 \item 
 $\kring$ a field (taken to be $\rat$ in Part \ref{part2});
 \item 
 $\fb$  the category of finite sets and bijections;
 \item 
  $\gr$ the category of finitely-generated free groups;
 \item 
 for $n\in \nat$, $\mathbf{n}:= \{1, \ldots , n \}$ so that $\{ \mathbf{n}\  | \ n \in \nat  \}$ is a  skeleton of $\fb$;
 \item 
 $\sym_n$ the symmetric group on $\mathbf{n}$.
 \end{itemize}

\subsection{Acknowledgements} The author is  grateful to Christine Vespa for her interest and for numerous comments and suggestions on this and related work.

He also thanks the anonymous referee for their careful reading of this work and for their comments and suggestions for improving the presentation. 

\bigskip
This work was partially supported by the ANR Project {\em ChroK}, {\tt ANR-16-CE40-0003}.
\tableofcontents

\part{The general case}

This first part of the paper sets up the general framework, working over a suitable operad, with $\kring$ a field of arbitrary characteristic.

 \section{The functor $\delta$}
\label{sect:delta} 
 
This section serves to introduce the category $\fopd$ of representations of the $\kring$-linear category $\cat \opd$ 
 associated to an operad $\opd$, together with the `shift' functor $\delta : \fopd \rightarrow \fopd$. (The reader is referred to \cite{P_analytic} for further details.)
 
%%%%%%%%%%%%%%%%%%%%%%%%%%%%%%%%%%%%%%%%%%%%%%%%%%%%%%%%%%%%%%%%%%%%%%%%%%%% 
\subsection{Background}

 Let $\opd$ be an operad in $\kring$-vector spaces and $\cat \opd$ be the associated PROP (see \cite[Section 5.4.1]{LV}, for example). Explicitly, $\cat \opd$ has set of  objects $\nat$ and 
\begin{eqnarray}
\label{eqn:cat_opd}
\cat \opd (m, n) 
= 
\bigoplus 
_{f \in \hom_{\mathsf{Set}}(\mathbf{m}, \mathbf{n})} 
\bigotimes _{i=1}^n 
\opd (|f^{-1}(i)|).
\end{eqnarray}

\begin{nota}
Denote by $\fopd$ the category of $\kring$-linear functors from $\cat \opd$ to $\kring$-vector spaces.
\end{nota}

\begin{rem}
\ 
\begin{enumerate}
\item 
If the operad $\opd$ is reduced (i.e., $\opd (0)=0$), then $\cat \opd (m,n)=0$ for $m<n \in \nat$. 
\item 
The category $\fopd$ is equivalent to the category of left $\cat\opd$-modules (see \cite{P_analytic}).
\item 
We will also use the category of {\em contravariant} $\kring$-linear functors from $\cat \opd$ to $\kring$-vector spaces; this is equivalent to the category of right $\cat\opd$-modules, denoted $\rmodO$.
\end{enumerate}
\end{rem}

The construction of $\cat \opd$ is natural with respect to the operad:  a morphism of operads $ \phi : \opd \rightarrow \ppd$ induces a morphism of PROPs $\cat \phi : \cat \opd \rightarrow \cat \ppd$. The $\kring$-linear functor between the underlying categories is the identity on objects.

\begin{nota}
For $\phi : \opd \rightarrow \ppd$ a morphism of operads, write $\phi^* : \fppd \rightarrow \fopd$ for restriction along  $\cat \phi : \cat \opd \rightarrow \cat \ppd$.  Explicitly, for $G \in \ob \fppd$, the functor $\phi^* G$ is given by $(\phi^* G)(n) = G(n)$, with morphisms acting accordingly.
\end{nota} 
 
\begin{lem}
\label{lem:conservative}
For $\phi : \opd \rightarrow \ppd$ a morphism of operads, the functor $\phi^* : \fppd \rightarrow \fopd$ is exact.  
 Moreover,  a morphism $f: G_1 \rightarrow G_2$ of $\fppd$ is:
\begin{enumerate}
\item 
an isomorphism if and only if $\phi^* f$ is an isomorphism; 
\item 
zero if and only if $\phi^* f$ is zero.
\end{enumerate}
\end{lem}

Yoneda's lemma gives the following: 
 
\begin{lem}
\label{lem:Yoneda}
For $m \in \nat$, $\cat \opd (m, -)$ corepresents evaluation on $m$: i.e.,  for $F \in \ob \fopd$, there is a natural isomorphism
$
\hom_{\fopd} (\cat \opd (m, -), F) \cong F(m).
$ 

Hence $\cat \opd (m, -)$ is projective in $\fopd$ and $\{ \cat \opd (m, -) \ | \ m \in \nat \}$ is a set of projective generators of $\fopd$.
\end{lem}
 
Recall that the unit operad $I$ is given by $I (n) =0$ for $n \neq 1$ and $I(1)= \kring$. 
The category $\smodug$ of left $\kring\fb$-modules is the category of $\kring$-linear functors from the $\kring$-linearization $\kring \fb$ of the  category  $\fb$ of finite sets and bijections  to $\kring$-vector spaces. 

\begin{rem}
In the literature, $\smodug$ is frequently known as the category of $\fb$-modules; it is equivalent to the category of functors from $\fb$ to $\kring$-vector spaces. Here modules are always understood  in the $\kring$-linear setting, to avoid conflicting usage.
\end{rem}
 
\begin{lem}
\label{lem:restrict_fb}
Restriction along the $\kring$-linear functor $\cat I \rightarrow \cat \opd$ induced by the unit $I \rightarrow \opd$ induces an exact restriction functor
 $
(-)\downarrow : 
\fopd \rightarrow \smodug.
$ 
In particular, for $F \in \ob \fopd$ and $n \in \nat$, $F(n)$ has a canonical underlying $\sym_n$-module structure.
\end{lem} 
 
 \begin{proof}
 The category $\cat I$ is equivalent to the $\kring$-linearization $\kring \fb$ of the category $\fb$. 
 \end{proof}
 
This restriction functor admits a section under suitable hypotheses:

\begin{prop}
\label{prop:alpha}
Suppose that the operad $\opd$ is reduced and that $\opd (1)= \kring$, generated by the unit. Then the restriction functor 
$ (-)\downarrow : 
\fopd \rightarrow \smodug
$ admits a section 
$
\alpha : 
\smodug 
\rightarrow 
\fopd
$ 
that sends a left $\kring\fb$-module $M$ to the functor $n \mapsto M(n)$ on which the morphisms of $\cat \opd (s,t)$ with $s \neq t$ act by zero.
\end{prop} 
 
\begin{proof}
The hypotheses imply that $\opd$ is augmented, with $\opd \rightarrow I$ the unique morphism of operads that is the identity in arity one. The functor $\alpha$ is induced by restriction along the augmentation. By construction, it is a section to the restriction functor.
\end{proof}

%%%%%%%%%%%%%%%%%%%%%%%%%%%%%%%%%%%%%%%%%%%%%%%%%%%%%%%%%%%%%%%%%%%%%%%%%%%%%%%%%% 
\subsection{Introducing $\delta$} 

The construction of $\delta$ below exploits the fact that $\cat \opd$ is a PROP, hence has a canonical symmetric monoidal structure (denoted here by $(\cat \opd, \tprop, 0)$) that corresponds to addition on the set of objects $\nat$. This gives the following:

 \begin{lem}
 \label{lem:shift}
 The symmetric monoidal structure $(\cat \opd, \tprop, 0)$  induces a faithful functor
 \[
 (\underline{\  }) \tprop 1  : \cat \opd \rightarrow \cat \opd
 \]
 that acts on objects by $n \mapsto n+1$, for $n \in \nat$.
 
This is natural with respect to the operad: for a morphism of operads $ \phi : \opd \rightarrow \ppd$, there is a commutative  diagram:
 \[
 \xymatrix{
 \cat \opd 
 \ar[r]^{\cat \phi}
 \ar[d]_{(\underline{\  }) \tprop 1}
 &
 \cat \ppd
 \ar[d]^{(\underline{\  }) \tprop 1}
 \\
 \cat \opd 
 \ar[r]_{\cat \phi} 
 &
 \cat \ppd. 
 }
 \]
 \end{lem}
 
 \begin{defn}
 \label{defn:delta}
 Let $\delta^\opd : \fopd \rightarrow \fopd$ be the functor  given by precomposition with $  (\underline{\  }) \tprop 1  : \cat \opd \rightarrow \cat \opd$. (When the operad $\opd$  is clear from the context, this will be denoted simply by $\delta$.)
 \end{defn}
 
 The following is immediate, recalling that $\mathbf{n}$ denotes the set  $ \{ 1, \ldots , n \}$ and $\sym_n$ is its automorphism group:
 
 \begin{prop}
 \label{prop:delta_exact}
 The functor $\delta : \fopd \rightarrow \fopd$ is exact and preserves coproducts in $\fopd$.  
 Explicitly, for $F \in \ob \fopd$ and $n \in \nat$, $\delta F(n) = F(n+1)$ as $\kring$-vector spaces,  
with underlying $\sym_n$-module
 \[
 \delta F (n) = F(n+1) \downarrow _{\sym_n}^{\sym_{n+1}},
 \]
 where the right hand side indicates the  restriction along $\sym_n \subset \sym_{n+1}$ induced by the inclusion $\mathbf{n} \subset \mathbf{n+1}$.
 \end{prop}
 
By Lemma \ref{lem:shift},  $\delta^\opd$ is natural with respect to the operad:

\begin{prop}
\label{prop:delta_nat_opd}
For $\phi : \opd \rightarrow \ppd$ a morphism of operads, the functors $\delta^\opd$ and $\delta^\ppd$ are compatible via $\phi^*$. Namely, for $G \in \ob \ppd$, there is a natural isomorphism
\[
\phi^* (\delta^\ppd G) \cong \delta^\opd (\phi^* G).
\] 
\end{prop}

%%%%%%%%%%%%%%%%%%%%%%%%%%%%%%%%%%%%%%%%%%%%%%%%%%%%%%%%%%% 
\subsection{The behaviour of $\delta$ on projectives and the convolution product for right $\cat\opd$-modules} 
\label{subsect:proj_right_modules}

 Understanding $\delta$ on the projective generators of the category $\fopd$  is of intrinsic interest. Moreover, this allows the functor $\delta$ to be analysed via Proposition \ref{prop:delta_universal} below, by reducing to the universal example. 
 
\begin{prop}
\label{prop:delta_proj}
For $m \in \nat$, $\delta \cat \opd (m, -)$ is projective. 
 Explicitly: 
\[
\delta \cat \opd (m, -)
= 
\bigoplus _{X \subset \mathbf{m}} \cat \opd (|X|, -) \otimes \opd (m-|X|),
\]
where the sum is taken over subsets $X$ of $\mathbf{m}= \{1, \ldots, m \}$.

If $\opd$ is a reduced operad, then the sum can be taken over proper subsets $X \subsetneq \mathbf{m}$.
\end{prop} 
 
 \begin{proof}
 It suffices to verify the explicit formula for $\delta \cat \opd (m, -)$, since the right hand side is projective, by Lemma \ref{lem:Yoneda} together with the fact that a coproduct of projectives is projective. The result can then be read off directly from  equation (\ref{eqn:cat_opd}).
 \end{proof}
 
\begin{rem}
The explicit expression for $\delta \cat \opd (m, -)$ can be written as the $\sym_m$-equivariant isomorphism:
\[
\delta \cat \opd (m, -)
\cong 
\bigoplus _{n \leq m} \big(\cat \opd (n, -) \otimes \opd (m-n)\big)\uparrow_{\sym_n \times \sym_{m-n}}^{\sym_m}.
\]
This can be expressed elegantly by using the convolution product $\conv$ on right $\kring\fb$-modules ($\kring$-linear functors from $\kring\fb\op$ to $\kring$-vector spaces) as explained below.
\end{rem}

Recall that, for $F, G$ right $\kring\fb$-modules, their convolution product $F \conv G$ is given  by
\[
F \conv G (Z) = \bigoplus _{Z = X \amalg Y} F(X) \otimes G(Y).
\]
When working with the skeleton of $\fb\op$, this is equivalent to  $(F \conv G )(m) = \bigoplus _{n \leq m}\big( F(n) \otimes G (m-n)\big)\uparrow_{\sym_n \times \sym_{m-n}}^{\sym_m}$.

\begin{nota}
Write $\rmodO$ for the category of right $\cat \opd$-modules and $\smodopug$ for the category of  right $\kring\fb$-modules, so that $\smodopug$ is equivalent to $\mathtt{Mod}_{\cat I}$.
\end{nota}

\begin{rem}
The category $\rmodO$ is equivalent to the category of right $\opd$-modules with respect to the operadic composition product $\circ$ (see \cite[Proposition 1.2.6]{KM}).
\end{rem}

One has the following counterpart of Lemma \ref{lem:restrict_fb}: 
 restriction along the functor $\kring \fb \cong \cat I \rightarrow \cat\opd$ induced by the unit of the operad induces an exact forgetful functor $\rmodO \rightarrow \smodopug$. 

The convolution product leads to a symmetric monoidal structure $(\rmodO , \rconv, \kring)$ such that the forgetful functor $\rmodO \rightarrow \smodopug$ is symmetric monoidal.  Under the equivalence between $\rmodO$ and right $\opd$-modules, the convolution product $\rconv$ identifies with  that of \cite[Proposition 1.6.3]{KM} for $\opd$-modules, introduced by Fresse  \cite{MR1617616}.

\begin{rem}
The notation $\rconv$ is introduced to avoid potential confusion, since we will also use the analogous convolution product for $\fopd$ (aka. left $\cat\opd$-modules), which will be denoted simply $\odot$ (see Section \ref{sect:mu_properties}).
\end{rem}

Now, $\delta \cat\opd$ has the structure of a $\cat\opd$-bimodule, induced by the canonical bimodule structure of $\cat\opd$. 
 Using the convolution product for right $\cat \opd$-modules, one can form
 \[
 \cat \opd \rconv \opd .
\]
This has the structure of a $\cat\opd$-bimodule: one has a left $\cat\opd$-module with values in right $\cat\opd$-modules given by 
\[
n \mapsto \cat\opd (-, n) \rconv \opd, 
\]
where the left $\cat\opd$-structure is derived from that of $\cat\opd$.

This allows the following reformulation of part of Proposition \ref{prop:delta_proj}:

\begin{prop}
\label{prop:delta_catopd_bimodule}
There is a natural isomorphism of $\cat\opd$-bimodules:
\[
\delta \cat\opd \cong \cat\opd \rconv \opd.
\]
\end{prop}

Now, tensor product over $\cat\opd$ yields a functor 
\[
- \otimes_{\cat\opd} - \ : \ \rmodO \times \fopd \rightarrow \kmod,
\]
using that $\fopd$ is equivalent to the category of left $\cat\opd$-modules (see \cite{P_analytic} for details). 
 Writing $\bimodO$ for the category of $\cat\opd$-bimodules, this induces:
\[
- \otimes_{\cat\opd} - \ : \ \bimodO \times \fopd \rightarrow \fopd.
\]

\begin{exam}
\label{exam:catopd_tensor}
For $F \in \ob \fopd$, $\cat\opd \otimes_{\cat\opd} F$ is naturally isomorphic to $F$. 

This can be fleshed out as follows. For $m \in \nat$, by construction of $\cat\opd \otimes_{\cat\opd} F$, there is a natural map $\cat \opd (m, -) \otimes_{\sym_m} F(m) \rightarrow \cat\opd \otimes_{\cat\opd} F$. Composing this with the above isomorphism gives the map $\cat \opd (m, -) \otimes_{\sym_m} F(m) \rightarrow F$ adjoint to the identity of $F(m)$ by Yoneda's lemma.
\end{exam}

One deduces the following:

\begin{prop}
\label{prop:delta_universal}
For $F \in \ob \fopd$, there is a natural isomorphism
\[
\delta F \cong (\delta \cat\opd) \otimes_{\cat\opd} F.
\]
\end{prop}

\begin{proof}
Clearly $\delta F$ is naturally isomorphic to $\delta ( \cat\opd \otimes_{\cat\opd} F)$, using the natural isomorphism exhibited in Example \ref{exam:catopd_tensor}. To conclude, it suffices to observe that there is a natural isomorphism 
\[
\delta ( \cat\opd \otimes_{\cat\opd} F)
\cong 
(\delta  \cat\opd ) \otimes_{\cat\opd} F.
\]
This follows directly from the definitions.
\end{proof}

\section{The full subcategory $\fopd^\mu$} 
\label{sect:mu} 
 
This section introduces the key players of the paper. The operation $\mut$ is defined and a criterion given for it to be a natural transformation (see Theorem \ref{thm:right_Leibniz}). Under the requisite right Leibniz hypothesis, this is  used to introduce the full category $\fopd^\mu \subset \fopd$.  
 
  %%%%%%%%%%%%%%%%%%%%%%%%%%%%%%%%%%%%%%%%%%%%%%%%%%%%%%%%%%%%%%%%%%%%%%%%%%%%%%%%%%%%%%%%%%%%%%%%
\subsection{The operation $\mut$ and the right Leibniz condition} 
 
The operad $\opd$ and $\mu \in \opd (2)$ are fixed throughout this subsection.

\begin{nota}
\label{nota:mu(n)}
For $n \in \nat$ and $i \in \mathbf{n}$, let $\mu_i (n) \in \cat \opd (n+1, n)$ be the morphism given by $\mu$ with entries $i, n+1$ and output $i$. Namely, with respect to the identification of equation (\ref{eqn:cat_opd}), this is the operation given by the set map $\mathbf{n+1} \rightarrow \mathbf{n}$ sending $j \mapsto j$ for $j <n+1$ and $n+1 \mapsto i$, using the identity in $\opd (1)$ for the fibres of cardinal one and $\mu \in \opd (2)$ for the remaining operation.

Set $\mu(n):= \sum_{i=1}^n \mu_i(n) \in \cat \opd (n+1, n)$ (by convention, $\mu(0)=0$).
 \end{nota}

 \begin{defn}
 \label{defn:mut}
 For $F \in \ob \fopd$ and $n \in \nat$, let  $\mut (n): \delta F (n) \rightarrow F(n)$ be the morphism of $\kring$-vector spaces  induced by $\mu (n) \in \cat \opd (n+1, n)$.
 \end{defn}
 
\begin{rem}
\label{rem:mut_right}
This definition of $\mut(n)$ involves a choice: namely the fixed `variable' $n+1$ acts via $\mu$ {\em on the right}, corresponding to the fact that $- \boxplus 1$ is chosen to define $\delta$.

This choice is compatible with that made in Definition \ref{defn:rho_rbar}, where group conjugation is taken to act on the right. It also corresponds to the fact that, in Section \ref{subsect:proj_flie_mu}, the tensor product adjoint action of a Lie algebra $\g$ on the $n$-fold tensor product $\g^{\otimes n}$, for $n \in \nat$, is taken to be on the right.
\end{rem} 

\begin{rem}
\label{rem:box_notation}
In \cite{P_Hcatlie}, where this material is used in the case $\opd= \lie$, the diagrammatic `box notation' is used to represent the morphisms $\mu(n)$. This notation is borrowed from the study of Jacobi diagrams (see  \cite[Example 3.2]{MR4321214} for example). 

Although this diagrammatic approach is not actually required here, the following example may elucidate the definition:
 $\mu (3)$ is represented by:

\begin{center}
 \begin{tikzpicture}[scale = .2]
 \draw (1,1) -- (1,-3);
 \draw (3,1) -- (3,-3);
 \draw (5,1) -- (5,-3);
\draw [rounded corners] (7,1) -- (7,-1) -- (6,-1);
 \draw [fill= lightgray] (0,-.5) -- (6,-.5) -- (6, -1.5)  -- (0, -1.5) -- cycle;
 \node at (7,-3) {,};
 \end{tikzpicture}
 \end{center}
which is shorthand for
\begin{center}
 \begin{tikzpicture}[scale = .2]
 \draw (1,1) -- (1,-3);
 \draw (3,1) -- (3,-3);
 \draw (5,1) -- (5,-3);
\draw [rounded corners] (7,1) -- (7,-1) -- (5,-1);
\draw [fill=black] (5,-1) circle (0.2);
\node at (10,-1) {$+$};
\draw (13,1) -- (13,-3);
 \draw (15,1) -- (15,-3);
 \draw (17,1) -- (17,-3);
 \draw [fill=white, white] (17,-1) circle (0.2);
\draw [rounded corners] (19,1) -- (19,-1) -- (15,-1);
\draw [fill=black] (15,-1) circle (0.2);
\node at (22,-1) {$+$};
\draw (25,1) -- (25,-3);
 \draw (27,1) -- (27,-3);
 \draw (29,1) -- (29,-3);
 \draw [fill=white, white] (29,-1) circle (0.2);
  \draw [fill=white, white] (27,-1) circle (0.2);
\draw [rounded corners] (31,1) -- (31,-1) -- (25,-1);
\draw [fill=black] (25,-1) circle (0.2);
\node at (31,-3) {,};
 \end{tikzpicture}
 \end{center}
 where $\bullet$ represents $[-,-] \in \lie (2)$. 
\end{rem}

The functor $(-)\downarrow$ used below is the restriction as in Lemma \ref{lem:restrict_fb}:
 
 \begin{lem}
 \label{lem:mu_sym_equivariant}
  For $F \in \ob \fopd$ and $n \in \nat$, the map $\mut (n) : \delta F (n) \rightarrow F(n)$ is $\sym_n$-equivariant.
  In particular, the morphisms $\mut(n)$ define a natural transformation 
\[
\mut_F\downarrow  \  : \  (\delta F)\downarrow \rightarrow F \downarrow
\]
of the underlying left $\kring\fb$-modules from $(\delta (-))\downarrow$ to $(-)\downarrow$.   
 \end{lem}
 
 \begin{proof}
Consider $i \neq j \in \{1, \ldots , n \}$ and let $\tau_{ij} \in \sym_n$ denote the associated transposition. Then one has
\[
\tau_{ij} \mu_j(n) \tau_{ij} = \mu_i (n)
\] 
whereas, if $k \not \in \{i, j\}$, $\tau_{ij} \mu_k(n) \tau_{ij} = \mu_k (n)$. This implies the stated equivariance.
 \end{proof}
 
In the following Definition,  $\bullet$ is used to denote composition in $\cat\opd$, to avoid potential confusion with the operadic composition product $\circ$.  We also use that  $\nu \in \opd (k)$ yields $\nu \in \cat \opd (k,1)$ and hence $\nu \tprop \id_1 \in \cat \opd (k+1,2)$. 
 
\begin{defn}
\label{defn:centrality_right_Leibniz}
For the given family $\{ \mu (n) \ | \ n \in \nat \}$, say that 
\begin{enumerate}
\item 
the operad $\opd$ satisfies the right Leibniz condition with respect to $\mu$ if, $ \forall n \in \nat$, $\forall \nu \in \opd (n)$, 
\[
\mu \bullet (\nu \tprop \id_1) = \nu \bullet  \mu (n)  \mbox{\  in $\cat \opd (n+1, 1) = \opd (n+1)$;}
\] 
\item 
the family $\{ \mu (n) \ | \ n \in \nat \}$ is central if, for all $n,s \in\nat$ and $\xi \in \cat\opd (n,s)$:
\[
\mu (s) \bullet (\xi \tprop \id_1)= \xi \bullet \mu (n)  \mbox{\  in $\cat \opd (n+1, s)$.} 
\]
\end{enumerate}
\end{defn}

These conditions are equivalent by the following:

\begin{lem}
\label{lem:centrality_Leibniz}
The family $\{ \mu (n) \ | \ n \in \nat \}$ is central if and only if  $\opd$ satisfies the right Leibniz condition with respect to $\mu$.

Moreover, the right Leibniz condition holds if and only if the equality $\mu \bullet (\nu \tprop \id_1) = \nu \bullet  \mu (n) $ holds for each $\nu \in \opd (n)$ belonging to  a set of generators of the operad. 
\end{lem}
 
\begin{proof}
By taking $s=1$, since $\mu(1) = \mu$,  it is clear that, if the family $\{ \mu (n) \ | \ n \in \nat \}$ is central, then $\opd$ satisfies the right Leibniz condition with respect to $\mu$.

The converse is a consequence of the fact that $\cat \opd$ is generated as a PROP by $\opd$. Hence the morphisms of  $\cat \opd$, considered as a $\kring$-linear category, are generated by the image of $\cat I \rightarrow \cat \opd$ and by elements of the form 
\[
\xi = \nu \tprop \id_{s-1} \in \cat \opd (n,s),
\]
for $\nu \in \opd (n-s+1)$, for the appropriate $n, s \in \nat$. 

Using the equivariance of $\mu (-)$ with respect to the symmetric groups given by Lemma \ref{lem:mu_sym_equivariant} and the aforementioned generating property, it suffices to establish the centrality identity when $\xi = \nu \tprop \id_{s-1} $ as above.   This follows readily from the case $s=1$ by the defining property of $\id_{s-1}$.

Finally, the above argument can be refined to considering a set of generators of the operad.
\end{proof} 
 
\begin{rem}
The box notation referenced in Remark \ref{rem:box_notation} can be useful in visualizing the argument of the proof of Lemma \ref{lem:centrality_Leibniz}. To illustrate this, consider $\nu \in \opd (2)$ a binary operation, represented in the diagrams below by $\circ$; the operation $\mu$ is represented by $\bullet$.

The corresponding right Leibniz condition is represented by
\begin{center}
 \begin{tikzpicture}[scale = .2]
 \draw (-7,1) -- (-7,0) -- (-6,-1) -- (-6,-4);
   \draw (-5,1) -- (-5,0) -- (-6,-1);
\draw [fill=white] (-6,-1) circle (.2);
\draw [rounded corners] (-4,1) -- (-4,-3) -- (-6,-3);
\draw [fill=black] (-6,-3) circle (.2);
\node at (-2,-1) {$=$};
\draw (1,1) -- (1,-2) -- (2,-3) -- (2,-4);
 \draw (3,1) -- (3,-2)-- (2,-3);
 \draw [fill=white] (2,-3) circle (.2);
\draw [rounded corners] (5,1) -- (5,-1) -- (4,-1);
 \draw [fill= lightgray] (0,-0.5) -- (4,-0.5) -- (4, -1.5)  -- (0, -1.5) -- cycle;
 \node at (5,-4) {.};
 \end{tikzpicture}
 \end{center}

Then, for example taking $s=3$ so that $\xi=\nu \boxplus\id_2$, the  centrality condition corresponds to:
\begin{center}
 \begin{tikzpicture}[scale = .2]
 \draw (-11,1) -- (-11,0) -- (-10,-1) -- (-10,-4);
   \draw (-9,1) -- (-9,0) -- (-10,-1);
\draw [fill=white] (-10,-1) circle (.2);
\draw (-8,1) -- (-8,-4); 
\draw (-6,1) -- (-6,-4); 
\draw [rounded corners] (-4,1) -- (-4,-2.5) -- (-5,-2.5);
 \draw [fill= lightgray] (-11,-2) -- (-5,-2) -- (-5, -3)  -- (-11, -3) -- cycle;
\node at (-2,-1) {$=$};
\draw (1,1) -- (1,-2) -- (2,-3) -- (2,-4);
 \draw (3,1) -- (3,-2)-- (2,-3);
 \draw (4,1) -- (4,-4); 
 \draw (6,1) -- (6,-4);
 \draw [fill=white] (2,-3) circle (.2);
\draw [rounded corners] (8,1) -- (8,-1) -- (7,-1);
 \draw [fill= lightgray] (0,-0.5) -- (7,-0.5) -- (7, -1.5)  -- (0, -1.5) -- cycle;
 \node at (8,-4) {.};
 \end{tikzpicture}
 \end{center}
This follows from the right Leibniz condition, since the contributions from the  application of the operation $\mu$ on the strands corresponding to $\id_2$ is the same on both sides.

This diagrammatic approach extends to treat the general case considered in Lemma \ref{lem:centrality_Leibniz}.
\end{rem}

\begin{rem}
\label{rem:leib}
If the hypotheses of Lemma \ref{lem:centrality_Leibniz} are satisfied, then $\mu$ itself must satisfy the right Leibniz condition. Written multiplicatively, this is the familiar 
$
(xy) z = (xz) y + x (yz).
$ 
The pair $(\opd, \mu \in \opd (2))$ is then equivalent to a morphism of operads $\leib \rightarrow \opd$ (where $\leib$ is the operad encoding right Leibniz algebras), together with the  right Leibniz condition for the elements of $\opd$ not in the image.
\end{rem}

\begin{thm}
\label{thm:right_Leibniz}
The morphisms $\mut(n)$ induce a natural transformation $\mut : \delta \rightarrow \mathrm{Id}$ if and only if $\opd$ satisfies the right Leibniz condition with respect to $\mu$.
\end{thm}

\begin{proof}
By definition, $\mut$ is a natural transformation if and only if, for all $\xi$ and for all $F \in \ob \fopd$, the following diagram commutes:
\[
\xymatrix{
\delta F (n) 
\ar@{=}[r]
\ar[d]_{\delta F (\xi) } 
&
F(n+1) 
\ar[r]^{\mut_F (n)} 
\ar[d]|{F (\xi \tprop \id_1) } 
&
F(n) 
\ar[d]^{F (\xi)}
\\
\delta F(s) 
\ar@{=}[r]
&
F(s+1)
\ar[r]
_{\mut_F (s)} 
&
F(s).
}
\] 

If $\opd$ satisfies the right Leibniz condition with respect to $\mu$ then, by Lemma \ref{lem:centrality_Leibniz}, the family $\{ \mu (n) \ | \ n \in \nat\}$ is central. It is clear that centrality implies the commutativity of the right hand square and the left hand square commutes by definition of $\delta$. This establishes the implication $\Leftarrow$. 

For $\Rightarrow$, one establishes centrality (and hence the right Leibniz condition) by using the method of the universal example (i.e., for each $n$, taking $F=  \cat \opd (n+1, -)$ and considering the image of the identity).
\end{proof}

\begin{exam}
The following pairs $(\opd, \mu \in \opd(2))$ satisfy the right Leibniz condition:
\begin{enumerate}
\item 
$\leib$ the operad of right Leibniz algebras with $\mu$ the generating operation; 
\item 
$\lie$, the Lie operad, with $\mu$ the generating operation; this corresponds (by Remark \ref{rem:leib}) to the usual morphism $\leib \rightarrow \lie$ of operads. 
\end{enumerate}
\end{exam}

%%%%%%%%%%%%%%%%%%%%%%%%%%%%%%%%%%%%%%%%%%%%%%%%%%%%%%%%%%%%%%%%%%%%%%%%%%%%%%%%%%%%%%%%%%%%%%%%
\subsection{Introducing $\fopd^\mu$}

Throughout this section, we assume that the pair $(\opd , \mu \in \opd(2))$ satisfies the conditions of Theorem \ref{thm:right_Leibniz} so that, 
for $F \in \ob \fopd$, one has the natural map $
\mut_F : \delta F \rightarrow F 
$
 in $\fopd$.

\begin{defn}
\label{defn:fopd^mu}
Let 
\begin{enumerate}
\item 
$\fopd^\mu \subset \fopd $ be the full subcategory  of functors $F$ for which $\mut_F=0$; 
\item
$(-)^\mu : \fopd \rightarrow \fopd$ be the functor
$
F \mapsto 
F^\mu := 
\mathrm{coker}\ \mut_F;
$
\item 
$\kappa_\mu : \fopd \rightarrow \fopd$ be the functor 
$
F \mapsto 
\kappa_\mu F := 
\ker \mut_F.
$
\end{enumerate}
\end{defn}

The following uses the functor $\alpha$ of Proposition \ref{prop:alpha}.

\begin{prop}
\label{prop:alpha_^mu}
Suppose that $\opd$ is reduced and that $\opd (1)= \kring$. Then the image of $\alpha : \smodug \rightarrow  \fopd$ lies in $\fopd^\mu$. In particular, the simple objects of $\fopd$ belong to $\fopd^\mu$.
\end{prop}

\begin{proof}
It is clear that the functor $\alpha$ commutes with coproducts; using this, one reduces readily to the case where $F \in \mathrm{image} (\alpha)$ is supported on a single $n \in \nat$ (i.e., $F(j)=0$ if $j \neq n$). Then $\delta F (j) = 0$ if $j \neq n-1$, by the identification given in Proposition \ref{prop:delta_exact}. It follows immediately that the natural transformation $\mut_F$ is zero. 

A functor $F$ of $\fopd$ is simple if and only if it is the image under $\alpha$ of a simple object of $\smodug$. Thus all simple objects belong to $\fopd^\mu$.
\end{proof}

The following result leads to the `universal' construction of objects of $\fopd^\mu$.

\begin{prop}
\label{prop:coker_mu}
For $F \in \ob \fopd$, $F^\mu$ belong to $\fopd^\mu$.  In particular, $(-)^\mu$ defines a functor $(-)^\mu \ : \ \fopd \rightarrow \fopd^\mu$.
\end{prop}

\begin{proof}
Naturality of the construction of $\mut$ gives the following commutative diagram:
\[
\xymatrix{
0
\ar[r]
&
\delta (\kappa_\mu F) 
\ar[r]
\ar[d]_{\mut_{\kappa_\mu F }}
&
\delta \delta F 
\ar[rr]^{\delta (\mut_F)}
\ar[d]|{\mut_{\delta F}}
&&
\delta F 
\ar[r]
\ar[d]|{\mut_F}
&
\delta (F^\mu)
\ar[r]
\ar[d]^{\mut_{F^\mu}}
&
0
\\
0
\ar[r]
&
\kappa_\mu F 
\ar[r]
&
\delta F 
\ar[rr]_{\mut_F}
&&
F 
\ar[r]
&
F^\mu 
\ar[r]
&
0,
}
\]
in which the rows are exact, using the exactness of $\delta$  (given by Proposition \ref{prop:delta_exact}) for the top row. 

The composite $\delta F \rightarrow F^\mu$ in the right hand commutative square is zero, since the maps around the bottom of the square appear in the bottom horizontal sequence. This implies that $F^\mu$ lies in $\fopd^\mu$, since $\delta F \rightarrow \delta (F^\mu) $ is surjective.
\end{proof}

\begin{rem}
\label{rem:ker}
The operation $\mu \in \opd (2)$ induces a natural transformation $\mu' : \delta \delta F \rightarrow \delta F$ via application of $\mu_{n+1}(n+2)$ (this does not require the right Leibniz condition on $\mu$). 
 The composite of the left hand square $\delta \kappa_\mu F \rightarrow \delta F$ in the above proof identifies as the composite:
\[
\delta \kappa_\mu F
\subset 
\delta \delta F 
\stackrel{\mu'}{\rightarrow} 
\delta F
.
\]
This is  non-zero in general. In particular,  $\kappa_\mu F$ does not in general belong to $\fopd^\mu$.
\end{rem}

%%%%%%%%%%%%%%%%%%%%%%%%%%%%%%%%%%%%%%%%%%%%%%%%%%%%%%%%%%%%%%%%%%%%%%%%%%%%%
\subsection{Fundamental properties of $\fopd^\mu$}

Some basic properties of the above constructions are established in this section.

\begin{prop}
\label{prop:closure}
\ 
\begin{enumerate}
\item 
The subcategory $\fopd^\mu$ is closed under the formation of subobjects, quotients and direct sums. 
\item 
If $0 \rightarrow F_1 \rightarrow F_2 \rightarrow F_3 \rightarrow 0$ is a short exact sequence with $F_1, F_3 \in \ob \fopd^\mu$, then $\mut_{F_2} : \delta F_2 \rightarrow F_2$ factors canonically as 
\[
\delta F_2 \twoheadrightarrow \delta F_3 \rightarrow F_1 \hookrightarrow F_2,
\]
where the first map is $\delta (F_2 \twoheadrightarrow F_3)$ and the last is the inclusion. In particular, $F_2$ belongs to $\fopd^\mu$ if and only if the morphism $\delta F_3 \rightarrow F_1$ is zero.
\end{enumerate}
\end{prop}

\begin{proof}
The stability statements are standard. For instance, consider a surjection $G \twoheadrightarrow Q$ in $\fopd$. This induces a commutative diagram
\[
\xymatrix{
\delta G \ar[r]^{\mut_G}
\ar@{->>}[d]
&
G
\ar[d]
\\
\delta Q 
\ar[r]_{\mut_Q}
&
Q,
}
\]
where the indicated epimorphism is given by the exactness of $\delta$. Hence, $Q$ lies in $\fopd^\mu$ if and only if the composite in the diagram is zero. In particular, this holds if $G$ belongs to $\fopd^\mu$.

The final statement is proved by using the morphism of short exact sequences provided by the naturality of $\mut$ (corresponding to the vertical maps) and the exactness of $\delta$:
\[
\xymatrix{
0 \ar[r]
&
\delta F_1 
\ar[r]
\ar[d]
&
\delta F_2
 \ar[r]
 \ar[d]
&
\delta
F_3
\ar[d]
\ar[r]&
0
\\
0 \ar[r]&
F_1 
\ar[r]
&
F_2
 \ar[r]
&
F_3
\ar[r]
&
0
.
}
\]
Under the hypotheses, the outer maps are zero, which leads to the required factorization.
\end{proof}

\begin{prop}
\label{prop:mu_left_adjoint}
The functor $(-)^\mu : \fopd \rightarrow \fopd^\mu$ is left adjoint to the inclusion $\fopd^\mu \hookrightarrow \fopd$. 
 Hence
\begin{enumerate}
\item 
$(-)^\mu$ is right exact;
 \item 
$(-)^\mu$ preserves projectives;
\item 
$\fopd^\mu$ has enough projectives. More precisely, $\{ \cat \opd (m, -) ^\mu \ | \  m \in \nat\}$ is a set of projective generators of $\fopd^\mu$.
\end{enumerate}

\end{prop}

\begin{proof}
To prove the adjunction statement, it suffices to show that a morphism $ f:  F\rightarrow G$, where $F \in \ob \fopd$ and $G \in \ob \fopd^\mu$, factors across the canonical surjection $F \twoheadrightarrow F^\mu$. By naturality of $\mut$, one has the commutative diagram of solid arrows:
\[
\xymatrix{
\delta F 
 \ar[r]^{\mut_F}
 \ar[d]_{\delta f}
 &
 F
 \ar@{->>}[r]
 \ar[d]_f
 &
 F^\mu = \mathrm{coker} \ \mut_F
 \ar@{.>}[ld]
 \\
 \delta G 
  \ar[r]_{\mut_G}
  &
  G,
}
\]
where the top row is exact. Since $G$ belongs to $\fopd^\mu$ by hypothesis, $\mut_G=0$; this gives the required factorization (indicated by the dotted arrow).

The remaining statements are then formal consequences of $(-)^\mu$ being left adjoint to an exact functor together with the fact that $\fopd$ has  set of projective generators  $\{ \cat \opd (m, -) \ | \ m \in \nat\}$, by Lemma \ref{lem:Yoneda}.
\end{proof}

The following result complements Proposition \ref{prop:mu_left_adjoint}:

\begin{prop}
\label{prop:kappa}
\ 
\begin{enumerate}
\item 
For $F \in \ob \fopd$,  the canonical inclusion $\kappa_\mu F \hookrightarrow \delta F$ is an isomorphism if and only if $F \in \ob \fopd^\mu$. 
\item 
The functor $\kappa_\mu : \fopd \rightarrow \fopd$ is left exact. 
\item 
For $0 \rightarrow F_1 \rightarrow F_2 \rightarrow F_3 \rightarrow 0$ an exact sequence in $\fopd$, the sequences associated to $\kappa_\mu$ and $(-)^\mu$ splice to give 
 an exact sequence
\[
0 \rightarrow \kappa_\mu F_1 \rightarrow \kappa_\mu  F_2 \rightarrow \kappa_\mu F_3 \rightarrow 
F_1^\mu  \rightarrow F_2^\mu \rightarrow F_3^\mu \rightarrow 0.
\]
\end{enumerate}
\end{prop}

\begin{proof}
The first statement follows directly from the definition of $\kappa_\mu$ and of $\fopd^\mu$. 

For the last two statements, consider the short exact sequence of complexes (of length two) given by applying $\mut$ to the given short exact sequence, as in the proof of Proposition \ref{prop:closure}. The six term exact sequence is given by the associated exact sequence in homology. In particular, this shows that $\kappa_\mu$ is left exact.
\end{proof}

%%%%%%%%%%%%%%%%%%%%%%%%%%%%%%%%%%%%%%%%%%%%%%%%%%%%%%%%%%%%%%%%%%%%%%%%%%%%%%%%%%%%%%%%
\subsection{The universal example}

Proposition \ref{prop:delta_universal} shows that, for $F \in \ob \fopd$ (viewed as a left $\cat\opd$-module), there is a natural isomorphism
\[
\delta F \cong (\delta \cat\opd) \otimes_{\cat\opd} F.
\]
This extends to show that $\mut : \delta \cat\opd \rightarrow \cat\opd$ provides the universal example for considering $\mut$:

\begin{prop}
\label{prop:mut_univ_example}
For $F \in \ob \fopd$, the following natural diagram commutes:
\[
\xymatrix{
(\delta \cat \opd )\otimes_{\cat\opd} F
\ar[d]_\cong 
\ar[rr]^{\mut \otimes \id_F} 
&&
\cat\opd \otimes_{\cat\opd} F
\ar[d]^\cong 
\\
\delta F 
\ar[rr]_{\mut} 
&&
F,
}
\]
where the left hand vertical isomorphism is given by Proposition \ref{prop:delta_catopd_bimodule} and the right hand one by Example \ref{exam:catopd_tensor}.
\end{prop}

\begin{proof}
This is essentially tautological when $F = \cat\opd (m, -)$, for some $m \in \nat$, and extends to considering naturality with respect to $m$ (i.e., the right $\cat\opd$-module structure of $\cat\opd$). From this, one deduces the general case.
\end{proof}

\section{Naturality with respect to the operad}
\label{sect:naturality}

In this short section we consider the naturality  of $ \fopd ^\mu$ with respect to suitable $(\opd, \mu)$. Whilst the main application   envisaged of the theory is to the case $\opd = \lie$, it is believed that exploiting the naturality may be a useful tool in studying this case.

Fix $\phi : \opd \rightarrow \ppd$ a morphism of operads so that the element $\mu \in \opd (2)$  yields $\phi \mu \in \ppd (2)$. One can thus construct the natural $\sym_n$-equivariant map:
\[
\widetilde{(\phi \mu)} (n) : \delta^\ppd G(n) \rightarrow G (n)
\]
for $G \in \ob \fppd$ and $n \in \nat$, as in Lemma \ref{lem:mu_sym_equivariant}.

In order to consider the associated natural transformations (cf. Theorem \ref{thm:right_Leibniz}), the following hypothesis is imposed throughout the section:

\begin{hyp}
\label{hyp:right_Leibniz_phi}
Both $(\opd, \mu \in \opd (2))$ and $(\ppd, \phi \mu \in \ppd (2))$ satisfy the  right Leibniz condition. 
\end{hyp}

\begin{prop}
\label{prop:compatibility_mu}
Via the natural isomorphism $\phi^* \delta^\ppd \cong \delta^\opd _{\phi^*} $ of Proposition \ref{prop:delta_nat_opd}, the  natural transformations
$ 
\phi^* \widetilde{(\phi \mu)} $ and $ \mut_{\phi^*}$ correspond.   
 More precisely, for $G \in \ob \fppd$, there is a natural commutative diagram in $\fopd$:
\[
\xymatrix{
\phi^* (\delta^\ppd G) 
\ar[rr]^{\phi^* \widetilde{(\phi \mu)}_G }
\ar[d]_\cong
&&
\phi^* G 
\ar[d]^=
\\
\delta^\opd (\phi^* G) 
\ar[rr]_{\mut_{\phi^* G}}
&&
\phi^* G. 
}
\]
\end{prop}

\begin{proof}
The natural transformation $\mut(n)$ is induced by $\sum_{i=1}^n \mu_i (n) \in \cat \opd (n+1, n)$. Applying $\phi : \cat \opd \rightarrow \cat \ppd$ gives the corresponding element of $\cat \ppd (n+1, n)$ associated to $\phi\mu \in \ppd (2)$. The result follows since the functor $\phi^* : \fppd \rightarrow \fopd$ is given by restriction along $\phi : \cat \opd \rightarrow \cat \ppd$.
\end{proof}

\begin{cor}
\label{cor:mu_phimu}
For $G \in \ob \fppd$, the following conditions are equivalent:
\begin{enumerate}
\item 
$G \in \ob \fppd^{\phi \mu}$; 
\item 
$\phi^* G \in \ob \fopd ^\mu$. 
\end{enumerate}
In particular, the functor $\phi^* : \fppd \rightarrow \fopd$ restricts to 
$ 
\phi^* : \fppd^{\phi \mu} \rightarrow \fopd^\mu.
$

Moreover, there are natural isomorphisms:
\begin{eqnarray*}
\phi^* (G^{\phi \mu}) & \cong & (\phi^* G)^\mu \\
\phi^* (\kappa_{\phi \mu } G) &\cong & \kappa_\mu (\phi^* G). 
\end{eqnarray*}
\end{cor}

\begin{proof}
By definition, $G$ belongs to $\fppd^{\phi \mu}$ if and only if the natural transformation $\widetilde{(\phi \mu)}_G : \delta^\ppd G \rightarrow G$ is zero.  By Lemma \ref{lem:conservative}, this is equivalent to the condition that $\phi^* \widetilde{(\phi \mu)}_G$ is zero. By Proposition \ref{prop:compatibility_mu}, this is equivalent to the condition that 
$\mut_{\phi^*G} : \delta^\opd \phi^*G \rightarrow \phi^* G$ is zero. Finally, by definition, the latter condition is equivalent to $\phi^* G$ belonging to $\fopd^\mu$, as required.

For the final isomorphisms, consider the exact sequence 
\[
0
\rightarrow 
\kappa _{\phi \mu} G 
\rightarrow 
\delta^\ppd G
\stackrel{\widetilde{\phi \mu }} {\rightarrow}
G
\rightarrow 
G^{\phi \mu}
\rightarrow 
0
\]
that constructs $\kappa _{\phi \mu} G $ and $G^{\phi \mu}$. Applying the exact restriction functor $\phi^*$, one deduces the result by the five-lemma, using  the commutative square of Proposition \ref{prop:compatibility_mu} and its vertical  isomorphisms.
\end{proof}

\begin{exam}
\label{exam:lie_leib_mu}
Consider the morphism of operads $\leib \rightarrow \lie$, using the canonical generator $\mu \in \leib (2)$.  Then, by Corollary \ref{cor:mu_phimu}, the restriction functor $\flie \rightarrow \fleib$ detects $\flie^\mu$. The restriction $\flie \rightarrow \fleib$ is a fully-faithful embedding and one has the following `pull-back' diagram of embeddings:
\[
\xymatrix{
\flie^\mu
\ar[r]
\ar[d]
&
\fleib^\mu 
\ar[d]
\\
\flie 
\ar[r]
&
\fleib.
}
\]
\end{exam}

\section{Stability of $\fopd^\mu$ under convolution}
\label{sect:mu_properties}

In this section, we consider the stability of $\fopd^\mu$ under the `convolution product' $\odot$ on $\fopd$.
For simplicity, we suppose that the operad $\opd$ is reduced and that the operad unit induces an isomorphism $\opd (1) \cong \kring$.

The convolution product $\odot$ on $\fopd$ is the analogue   of $\rconv$ for right $\cat\opd$-modules that was used in Section \ref{subsect:proj_right_modules}.

\begin{prop}
\cite{P_analytic}
\label{prop:left_convolution}
There is a  convolution product $\conv$ on $\fopd$ that  yields a symmetric monoidal structure $(\fopd, \conv , \kring)$.
This satisfies the following properties:
\begin{enumerate}
\item 
For $\opd=I$, via the identification $\f_I \cong \smodug$, this is  the usual convolution product of left $\kring\fb$-modules.
\item 
The forgetful functor $\fopd \rightarrow \f_I \cong \smodug$ induced by the operad unit $I \rightarrow \opd$ is symmetric monoidal.
\item 
The bifunctor $\conv$ is exact with respect to both variables.
\end{enumerate}
\end{prop}

To prove the results of this section, we require to recall the construction. For this it is convenient to use the standard `coordinate-free' approach, extending the objects of $\cat\opd$ to allow arbitrary finite sets by  using  Kan extension; likewise in considering the operad $\opd$. 

Consider left $\cat \opd$-modules $F, G$; then, evaluated on a finite set $Z$, one has 
\[
(F \conv G) (Z) = \bigoplus _{X \amalg Y = Z} F(X) \otimes G(Y)
\]
where the sum is indexed over ordered decompositions of $Z$ into two subsets (possibly empty). We must  specify the action of $\cat \opd (Z, W)$, for finite sets $Z$, $W$. Under the hypotheses on $\opd$ in force, this decomposes into components indexed by {\em surjective} set maps $f : Z \twoheadrightarrow W$: 
\[
\cat \opd (Z, W) 
\cong 
\bigoplus_{f : Z \twoheadrightarrow W}
\cat \opd (Z, W) _f 
\]
where $\cat \opd (Z, W) _f := \bigotimes_{w \in W} \opd (f^{-1}(w))$. (This is the `coordinate-free' analogue of (\ref{eqn:cat_opd}).) 
 Thus it suffices to consider the action of $\cat \opd (Z, W) _f $  on $F(X) \otimes G(Y)$, for a fixed decomposition $Z = X \amalg Y$.

If $f(X) \cap f(Y) \neq \emptyset$, then this component acts by zero. Otherwise $W = f(X) \amalg f(Y)$ and, with respect to these identifications,  $f$ is given as the disjoint union of $f|_X : X \twoheadrightarrow f(X)$ and $f|_Y : Y \twoheadrightarrow f(Y)$. This induces a decomposition:
\[
\cat \opd (Z, W)_f
\cong 
\cat \opd (X, f(X))_{f|_X}
\otimes 
\cat \opd (Y, f(Y))_{f|_Y}.
\]

Using this, the left $\cat \opd$-structures of $F$ and $G$ yield the action of $\cat \opd (Z, W)_f$:
\[
F(X) \otimes G(Y) \rightarrow F(f(X)) \otimes G(f(Y) ) \subset (F \odot G) (W). 
\]
This describes the $\cat \opd$-structure of $F \odot G$.

\begin{lem}
\label{lem:delta_mu_conv}
For $F , G \in \ob \fopd$, 
\begin{enumerate}
\item 
there is a natural isomorphism $\delta (F \conv G) \cong (\delta F) \conv G \ \oplus \ F \conv (\delta G)$; 
\item 
with respect to this isomorphism, the map $\mut_{F\conv G} : \delta (F \conv G) \rightarrow F \conv G$ identifies with the sum of the maps 
$\mut_F \conv \id_G $ and $\id_F \conv  \mut_G$. 
\end{enumerate}
\end{lem}

\begin{proof}
In the coordinate-free approach, for $M$ a left $\cat \opd$-module, $\delta M$ is given by $\delta M (Z):= M(Z_+)$ where $Z_+ = Z\amalg \{*\}$, with the  action of morphisms via
\[
\cat \opd (Z, W) \rightarrow \cat \opd (Z_+, W_+),
\]
treating $(-)_+$ as the analogue of $- \boxplus 1$ of Lemma \ref{lem:shift}.

Then one has
$$
\delta (F \conv G) (Z) = \bigoplus _{U \amalg V = Z_+} F(U) \otimes G(V).
$$
Clearly, one must either have $+ \in U$ or $+ \in V$. It follows that the right hand side can be written as 
$$
\bigoplus _{U' \amalg V' = Z} \big (\delta F(U') \otimes G(V') \ \oplus \  F(U') \otimes \delta G(V')\big), 
$$
where the first term corresponds to $+ \in U$ and the second to $+ \in V$. This expression identifies with $\big((\delta F) \conv G \ \oplus \ F \conv (\delta G)\big) (Z)$. 

It remains to check that this identification is compatible with the action of morphisms of $\cat \opd$. This follows from the explicit description of this action given after the statement of Proposition \ref{prop:left_convolution}; this is a straightforward verification using the fact that the added basepoint $+$ does not intervene in the action of morphisms. 

The second statement is proved similarly. However, in this case, the basepoint $+$ is a key player. One works with the analogue of the map $\mu (n)$ of Notation \ref{nota:mu(n)}. This is the map 
\[
\mu (Z) \in \cat \opd (Z_+, Z)
\]
defined as the sum of maps $\mu_z \in \cat \opd (Z_+, Z)$, for $z \in Z$, where $\mu_z$ is associated with the surjection $Z_+ \twoheadrightarrow $ that is the identity on $Z$ and sends $+$ to $z$, mimicking the definition of $\mu_i (n)$. 

Then $\mu(Z)$ induces the analogue of $\tilde{\mu}(n)$:
\[
\tilde{\mu} (Z) : \delta (F \odot G) (Z) \rightarrow (F \odot G) (Z).
\]

To identity this, consider the contribution from  $\mu_z$ acting on the component $\delta F(U') \otimes G(V') \ \oplus \  F(U') \otimes \delta G(V')$ of $\delta (F \odot G) (Z)$, using the above identifications, where $U' \amalg V' = Z$. 
 If $z\in U'$, then $\mu_z$ acts via its action $\delta F(U')\rightarrow F(U')$ on the first direct summand and by zero on the second direct summand. If $z \in V'$, then $\mu_z$ acts by zero on the first summand and via its action $\delta G(V')\rightarrow G(V')$ on the second. Thus, the contribution from $\mu_z$ gives
 \[
 \delta F(U') \otimes G(V') \ \oplus \  F(U') \otimes \delta G(V')
 \rightarrow 
 F(U') \otimes G(V') \subset (F \odot G) (Z)
 \]
described by the above.

Then, summing over $z \in Z$, this gives the sum of the maps
\begin{eqnarray*}
\tilde{\mu} (U') \otimes \id &:& \delta F (U') \otimes G(V') \rightarrow F(U') \otimes G(V') \\
\id \otimes  \tilde{\mu} (V') &:& F (U') \otimes \delta  G(V') \rightarrow F(U') \otimes G(V').
\end{eqnarray*}

Putting these identifications together, one obtains the result.
\end{proof}

This gives the following stability result:

\begin{prop}
\label{prop:stable_conv}
The convolution product $(\fopd, \conv, \kring)$ restricts to a symmetric monoidal structure 
 $(\fopd^\mu, \conv, \kring)$.
 
Moreover, for $F, G \in \ob \fopd$, the canonical morphism 
\[
(F \conv G)^\mu 
\rightarrow 
F^\mu \conv G^\mu 
\]
is an isomorphism.

\end{prop}

\begin{proof}
For the first statement, it suffices to prove that, if $F, G \in \ob \fopd^\mu$, then $F\conv G$ is also. 
The hypothesis on $F, G$ is equivalent to the morphisms $\mut_F$ and $\mut_G$ both being zero.  
By Lemma \ref{lem:delta_mu_conv}, this implies that $\mut_{F \conv G}=0$, as required.

Now consider the general case, with $F, G \in \ob \fopd$. The convolution product of the  morphisms $\mut_F$ and $\mut_G$ give 
 the commutative diagram:
\[
\xymatrix{
\delta F \conv \delta G
\ar[rr]^{\id_{\delta F} \conv \mut_G}
\ar[d]_{\mut_F \conv \id_{\delta G}}
&&
\delta F \conv  G
\ar[d]^{\mut_F \conv \id_{G}}
\\
F \conv \delta G
\ar[rr]_{\id_{F} \conv \mut_G}
&&
F \conv G.
}
\]
This has an associated double complex (placing $F \conv G$ in homological degree zero) and $(F\conv G)^\mu$ is calculated as the degree zero homology of the associated total complex, by Lemma \ref{lem:delta_mu_conv}.

By filtering  (taking the right hand column as a subcomplex) and considering the associated (baby) spectral sequence, the homology in degree zero is isomorphic to the cokernel of the  morphism 
\[
\xymatrix{
F^\mu \conv \delta G
\ar[rr]^{\id_{F^\mu} \conv \mut_G} 
&&
F^\mu \conv G,
}
\] 
where exactness of $-\conv \delta G$ and $ -\conv  G$ has been used to identify the terms. Then, by exactness of $F^\mu \conv -$, the cokernel is isomorphic to $F^\mu \conv G^\mu$, as required. 
\end{proof}

\part{The Lie case}
\label{part2}

We now specialize to the case $\opd= \lie$ over $\kring=\rat$, taking  $\mu \in \lie(2)$ to be the generator corresponding to $[-,-]$.  The main result explains the relationship between $\flie^\mu$ and the category of outer functors on $\gr\op$, where $\gr$ is the category of finitely-generated free groups.

%\input{lie_case}
 %%%%%%%%%%%%%%%%%%%%%%%%%%%%%%%%%%%%%%%%%%%%%%%%%%%%%%%%%%
\section{Outer functors and the statement of the main result}
\label{sect:lie_case}

After reviewing the definition of the category $\foutanQ$ of analytic outer functors on $\gr\op$, the main result is stated in Theorem \ref{thm:outer_mu}, relating this category to $\flie^\mu$. Some immediate consequences are given in Section \ref{subsect:consequences}; the proof of the Theorem is postponed until Section \ref{sect:proof_lie_outer}.

%%%%%%%%%%%%%%%%%%%%%%%%%%%%%%%%%%%%%%%%%%%%%%%%%%%%%%%%%%%%%
\subsection{Functors on $\gr\op$} 

Let $\gr$ be the category of finitely-generated free groups, considered as a full subcategory of the category of groups. This has skeleton given by the groups $\zed^{\star n}$, for $n \in \nat$,  so that one can consider the category $\fcatQ[\gr\op]$ of functors from $\gr\op$ to $\rat$-vector spaces. 

Using the free product $\star$ on $\gr$, one has the notion of a polynomial functor on $\gr\op$ (cf. \cite{MR3340364}). This in turn allows one to consider analytic functors on $\gr\op$, i.e., those that are the colimit of their sub polynomial functors. (See  \cite{P_analytic} for more details; part of the theory is revisited in \cite{P_malcev}.)

\begin{nota}
Denote by 
\begin{enumerate}
\item 
$\f_d (\gr\op ; \rat)  \subset \fcatQ [\gr\op]$ the full subcategory of polynomial functors of degree at most $d\in \nat$;  
\item 
$\f_\omega (\gr\op; \rat) \subset \fcatQ [\gr\op]$ the full subcategory of analytic functors.
\end{enumerate}
\end{nota}

The category of analytic contravariant functors is modelled by representations of $\cat \lie$:

\begin{thm}
\cite{P_analytic}
\label{thm:grop_analytic}
The category $\f_\omega (\gr\op; \rat) $ is equivalent to $\flie$.
\end{thm}

The equivalence is made explicit as follows, using that $\flie$ can be considered as the category of left $\cat \lie$-modules. The presentation given here is  equivalent to that of \cite{P_analytic}, but uses the construction of the `twisting bimodule' adopted in \cite{P_malcev}.

\begin{nota}
Let $\uass$ denote the  operad encoding unital associative algebras, with morphism of operads $\lie \rightarrow \uass$ encoding the underlying Lie algebra of an associative unital algebra. 
\end{nota}

Now, $\lie$ can be considered as a Lie algebra in the category $\rmod$ of right $\cat\lie$-modules, equipped with its symmetric monoidal structure $(\rmod, \rconv, \rat)$ (see \cite[Observation 9.1.3]{MR2494775}, for example). Mimicking the classical construction of the universal enveloping algebra of a Lie algebra, working in $(\rmod, \rconv, \rat)$, one can form the unital associative algebra $U \lie$ in $\rmod$. The morphism $\lie \rightarrow \uass$ is a morphism of Lie algebras in $\rmod$ and the induced morphism 
\[
U\lie \stackrel{\cong}{\rightarrow} \uass
\]
is an isomorphism of unital associative algebras in $\rmod$.  Moreover, $\lie \rightarrow \uass$ corresponds to the canonical inclusion $\lie \hookrightarrow U \lie$.

The advantage of this viewpoint on $\uass$ is that, as for the classical universal enveloping algebra of Lie algebras, $U\lie$ has the structure of a cocommutative Hopf algebra in $\rmod$. This allows one to apply the exponential functor construction $\Phi$, using the notation of \cite{PV}, to form the functor $\Phi (U\lie)$ from $\gr\op$ to $\rmod$. (See \cite{PV} or the review in \cite{P_analytic} for some more details on $\Phi$.)

In the case at hand, by the definition of $\Phi$,
\[
\Phi (U \lie) (\zed^{\star n}) = (U \lie)^{\rconv n},
\]
with action of morphisms of $\gr\op$ determined by the Hopf algebra structure of $U\lie$. 

We can consider $\Phi (U \lie)$ as a left $\rat\gr\op$, right $\cat\lie$ bimodule. The equivalence of Theorem \ref{thm:grop_analytic} is given by the functor 
\[
\Phi (U\lie) \otimes_{\cat \lie} - \ : \ \flie \stackrel{\cong}{\rightarrow} \f_\omega(\gr\op; \rat).
\] 

\begin{rem}
The left $\rat\gr\op$, right $\cat\lie$ bimodule $\Phi (U\lie)$ is isomorphic to $\gropcatuass$ used in \cite{P_analytic}.  By construction, for $M\in \ob \flie$, when  evaluated on $\zed^{\star n}$, the functor $\Phi (U\lie) \otimes_{\cat \lie} M$ gives 
$$\cat\uass (-,n) \otimes_{\cat\lie} M,$$
 where the right $\cat\lie$-module structure on $\cat\uass$ is given by restriction of the canonical right $\cat\uass$-module structure along the induced $\cat \lie \rightarrow \cat \uass$.
\end{rem}

\begin{exam}
\label{exam:cat_lie_m}
For $m \in \nat$, $\cat \lie (m, -)$ is a $\cat \lie$-module; by  Lemma \ref{lem:Yoneda}, it is projective and 
$$
\{ \cat \lie (m, -) \ | \ m \in \nat \}$$
 is a set of projective generators for $\cat \lie$. 

Under the equivalence of Theorem \ref{thm:grop_analytic}, $\cat \lie (m, -)$ corresponds to $\Phi (U\lie) (m)$, the $\rat\gr\op$-module obtained by evaluating the bimodule $\Phi (U\lie)$ on $m \in \ob \cat\lie$. Hence 
$\{ \Phi (U \lie) (m) \ | \ m \in \nat\}$ is a set of projective generators of $\f_\omega (\gr \op; \rat)$.

In the notation of \cite{P_analytic}, $\Phi (U \lie ) (m)$ identifies as $\gropcatuass (m,-)$. 
\end{exam}

One has the following useful characterization of $\cat \lie (m, -)$:

\begin{prop}
\label{prop:proj_cover_flie}
For $m \in \nat$, $\cat \lie (m, -)$ is the projective cover of $\rat \sym_m$ in $\flie$, where $\rat\sym_m$ is considered as a $\cat\lie$-module supported on $\mathbf{m}$.  
\end{prop}

\begin{proof}
That $\cat\lie (m, -)$ is projective is given by Lemma \ref{lem:Yoneda}.

Now, $\cat\lie (m, n)$ is zero if $n>m$ and, for $n=m$, is isomorphic to $\rat \sym_m$. It follows that there is a surjection
\[
\cat \lie (m, - ) \twoheadrightarrow \rat \sym_m
\]
with kernel that is an isomorphism when evaluated on $m$. 

That this exhibits $\cat \lie (m, -)$ as the projective cover of $\rat \sym_m$ is then a straightforward consequence of the Yoneda lemma.
\end{proof}

%%%%%%%%%%%%%%%%%%%%%%%%%%%%%%%%%%%%%%%%%%%%%%%%%%%%%%%%%%%%%%%%
\subsection{Outer functors}

The following category of outer functors was introduced in \cite{PV}.

\begin{defn}
\label{defn:fout}
Let 
\begin{enumerate}
\item
$\foutQ \subset \fcatQ[\gr\op]$ be the full subcategory with objects functors $F  \in \ob \fcatQ[\gr\op]$ such that, for all $n \in \nat$, the subgroup of inner automorphisms $\mathrm{Inn} (\zed^{\star n}) \subset \mathrm{Aut} (\zed^{\star n})$ acts trivially on $F (\zed^{\star n}) $. 
\item 
$\foutanQ \subset \f_\omega (\gr\op; \rat)$ be the full subcategory of analytic outer functors.
\end{enumerate}
\end{defn}

Below, $\fgrp$ is used to denote an object of $\gr$ (and, hence, of $\gr\op$).

\begin{nota}
Denote by
\begin{enumerate}
\item 
$\tau : \fcatQ[\gr\op] \rightarrow \fcatQ[\gr\op]$ the shift functor defined on $F \in \ob \fcatQ[\gr\op]$ by
$(\tau F ) (\fgrp):= F ( \fgrp \star \zed)$;
\item 
$\tbar : \fcatQ[\gr\op] \rightarrow \fcatQ[\gr\op]$  the reduced shift functor that is defined by the canonical splitting 
$\tau F \cong F \oplus \tbar F$ induced by the group morphisms $\{e \} \rightarrow \zed \rightarrow \{e \}$.
\end{enumerate}
\end{nota}

\begin{defn}
\label{defn:rho_rbar}
Let 
\begin{enumerate}
\item 
$\rho : \tau \rightarrow \id$ be the natural transformation of functors on $\fcatQ[\gr\op]$  induced by the `universal conjugation' map $\fgrp \rightarrow  \fgrp \star \zed$, $g \mapsto x^{-1} g x$, for $x$ the generator of the subgroup $\zed \subset \fgrp \star \zed$;
\item
$\rbar :  \tbar \rightarrow \id$ be the composite natural transformation 
$ 
\tbar  \hookrightarrow \tau  \stackrel{\rho}{\rightarrow} \id.
$
\end{enumerate}
\end{defn}

\begin{rem}
\label{rem:conj_right}
The  definition of $\rho$ involves the choice that the conjugation acts on the right; this is consistent with the choice  stressed in Remark \ref{rem:mut_right}.
\end{rem}

The following is a standard  consequence of the definition of polynomial functors on $\gr\op$:

\begin{lem}
\label{lem:tbar_analytic}
The functor $\tbar : \fcatQ[\gr\op] \rightarrow \fcatQ[\gr\op]$ restricts to 
$ 
\tbar : \f_\omega (\gr\op; \rat)
\rightarrow 
\f_\omega (\gr\op; \rat).
$
\end{lem}

The category $\foutQ$ identifies in terms of $\rbar$ as follows:

\begin{prop}
\label{prop:fout_equivalent}
The category $\foutQ$ (respectively $\foutanQ$) is equivalent to the full subcategory of $\fcatQ[\gr\op]$ (resp. $\f_\omega (\gr\op; \rat)$) of functors for which $\rbar=0$.
\end{prop}

\begin{proof}
 Clearly the statement for analytic functors follows from that for $\foutQ$. The latter is essentially established in \cite[Section 11]{PV}, where the functor $\Omega : \f (\gr\op; \rat)\rightarrow \foutQ $ left adjoint to the inclusion is defined.   By construction, a functor $F$ is in $\foutQ$ if and only if the adjunction unit $F \rightarrow \Omega F$ is an isomorphism. 
  
Now, by \cite[Proposition 11.8]{PV} (where different notation is used), the functor $\Omega$ is naturally isomorphic to the cokernel of $\rbar$. The result follows. 
\end{proof}

On restricting to analytic functors, Proposition \ref{prop:fout_equivalent} (and its proof) has the following Corollary, the analogue of Proposition \ref{prop:mu_left_adjoint}:

\begin{cor}
\label{cor:left_adjoint_out}
The inclusion $\foutanQ \subset \f_\omega (\gr\op; \rat)$ has left adjoint given by $ F\mapsto \mathrm{coker} \rbar_F$.
\end{cor}

%%%%%%%%%%%%%%%%%%%%%%%%%%%%%%%%%%%%%%%%%%%%%%%%%%%%%%%%%%%%%%%%%%%%%%
\subsection{Statement of the main result}
\label{subsect:consequences}

The following counterpart of Theorem \ref{thm:grop_analytic} for outer analytic functors is the main result of Part \ref{part2}; it will be proved in Section \ref{sect:proof_lie_outer}, where it is reformulated slightly more explicitly as Theorem \ref{thm:outer_mu_explicit}.

\begin{thm}
\label{thm:outer_mu}
Under the equivalence  of Theorem \ref{thm:grop_analytic}, the full subcategory $\foutanQ \subset \f_\omega (\gr\op; \rat)$ identifies with $\flie^\mu \subset \flie$.
\end{thm}

Known structure of $\foutan$ is reflected across the equivalence of Theorem \ref{thm:outer_mu}, as illustrated by the following examples.

\begin{exam}
By Proposition \ref{prop:alpha_^mu}, the functor $\alpha : {}_{\rat \fb}\mathtt{Mod}
\rightarrow 
\flie$ maps to $\f_\lie^\mu$, giving the factorization 
\begin{eqnarray}
\label{eqn:ab_flie_mu}
{}_{\rat \fb}\mathtt{Mod}
\stackrel{\alpha}{\rightarrow} 
\flie^\mu 
\hookrightarrow 
\flie.
\end{eqnarray}

Write $\ab$ for the full subcategory of the category of groups with objects finitely-generated free abelian groups, so that one has $\f_\omega (\ab\op; \rat) $, the category of analytic functors on $\ab\op$, defined analogously to $\f_\omega (\gr\op; \rat)$.  Restriction along the abelianization functor $\gr \rightarrow \ab$ induces  $\f_\omega (\ab\op; \rat) \hookrightarrow \f_\omega (\gr\op; \rat)$ and this clearly factors across the full subcategory of outer functors.

Via Theorem \ref{thm:outer_mu}, (\ref{eqn:ab_flie_mu}) corresponds to 
\[
\f_\omega (\ab\op; \rat) \hookrightarrow \foutanQ \hookrightarrow \f_\omega (\gr\op; \rat).
\]
\end{exam}

\begin{exam}
In \cite{P_analytic}, it is shown that the convolution product $\conv$ on $\flie$ corresponds, via the equivalence of Theorem \ref{thm:grop_analytic}, to the usual tensor product on $\fcatQ[\gr\op]$, restricted to the full subcategory of analytic functors. 

The stability of $\flie^\mu$ under the convolution product $\conv$ given by Proposition \ref{prop:stable_conv} corresponds to the fact that $\foutanQ$ is stable under the tensor product on $ \f_\omega (\gr\op; \rat) \subset \fcatQ[\gr\op]$.
\end{exam}

%%%%%%%%%%%%%%%%%%%%%%%%%%%%%%%%%%%%%%%%%%%%%%%%%%%%%%%%%%%%%%%%
\subsection{From Lie algebras to projective generators}
\label{subsect:proj_flie_mu}

By Proposition \ref{prop:mu_left_adjoint}, 
\begin{eqnarray}
\label{eqn:proj_gen_flie_mu}
\{ \cat \lie (m, -)^\mu \ | \ m \in \nat \}
\end{eqnarray}
 is a set of projective generators of $\flie^\mu$. The purpose of this section is to give a description of these in terms of Lie algebra homology. 

We first note the following analogue of Proposition \ref{prop:proj_cover_flie}, giving a convenient characterization of the objects $\cat \lie (m, -)^\mu$, for $m \in \nat$:

\begin{prop}
\label{prop:proj_cover_flie_mu}
For $m \in \nat$, $\cat \lie (m, -)^\mu$ is the projective cover of $\rat \sym_m$ in $\flie^\mu$.  
\end{prop}

\begin{proof}
This is a consequence of the fact that $\cat \lie (m, -)$ is the projective cover of $\rat \sym_m$ in $\flie$, as shown in  Proposition \ref{prop:proj_cover_flie}. Namely, the surjection $\cat \lie (m, -) \twoheadrightarrow \rat \sym_m$ in $\flie$ induces  a surjection $\cat \lie (m, -)^\mu \twoheadrightarrow \rat \sym_m$ in $\flie^\mu$, since $(-)^\mu$ is right exact; this exhibits $\cat \lie (m, -)^\mu$ as the projective cover of $\rat \sym_m$ in $\flie^\mu$, as in the proof of Proposition \ref{prop:proj_cover_flie}.
\end{proof}

For $\g$ a Lie algebra, one has the associated left $\cat\lie$-module $\underline{\g}$ given by $\underline{\g} (n) = \g ^{\otimes n}$, with morphisms acting via the Lie algebra structure of $\g$ (see \cite[Proposition 5.4.2]{LV}, for example). 

The following interprets $(\underline{\g})^\mu$ in terms of zero degree Lie homology (i.e., coinvariants), where we consider $\g$ as a right $\g$-module for the adjoint action and equip  $\g^{\otimes n}$ with the tensor product structure, for each $n \in \nat$. 

\begin{prop}
\label{prop:underline_g_mu}
For $\g$ a Lie algebra, $(\underline{\g})^\mu \in \ob \flie^\mu$ has values
$
n \mapsto H_0 (\g; \g^{\otimes n}) 
$, 
with $\cat \lie$-module structure induced from that of $\underline{\g}$ via the canonical quotient maps $\g^{\otimes n} \twoheadrightarrow  H_0 (\g; \g^{\otimes n})$.

This identification is natural with respect to the Lie algebra $\g$.
\end{prop}

\begin{proof}
By definition, $(\underline{\g})^\mu$ is  the cokernel of $\mut_{\underline{\g}} : \delta \underline{\g} \rightarrow \underline{\g}$. Evaluated on $n \in \nat$, by definition of $\mut$ and of the structure of $\underline{\g}$, this identifies as 
\[ 
\g^{\otimes n+1} = (\g ^{\otimes n}) \otimes \g 
\rightarrow 
\g^{\otimes n}
\]
given by the above right $\g$-module structure  on $\g^{\otimes n}$. This establishes the first statement.

The naturality with respect to $\g$ is clear.
\end{proof}

\begin{nota}
\label{nota:H0_underline_g_mu}
Denote by  $H_0(\g; \underline{\g})$ the left $\cat \lie$-module $
n \mapsto H_0 (\g; \g^{\otimes n}) 
$
 of  Proposition \ref{prop:underline_g_mu}. 
\end{nota}

We now apply this to study the projective generators of $\flie^\mu$.  For this, since we are working over $\rat$, we  can exploit the Schur correspondence (see \cite[Section 3.4]{P_analytic} for a brief survey and \cite[Appendix I.A]{MR3443860} for further details). For $G$ a right $\rat\fb$-module, the associated Schur functor is the functor on $\fvs$, the category of finite-dimensional $\rat$-vector spaces, given by 
\begin{eqnarray}
\label{eqn:G_schur}
V \mapsto G(V) := \bigoplus_{\ell \in \nat} G(\ell) \otimes_{\sym_\ell} V^{\otimes \ell}.
\end{eqnarray}

For given  $m \in \nat$, the $\sym_m$-module $G(m)$ can be recovered from the Schur functor $V \mapsto G(V)$ by 
\begin{eqnarray}
\label{eqn:G_m}
G(m) \cong 
\nt (V^{\otimes m} , G(V) ),
\end{eqnarray}
where the notation $\nt$ is shorthand for natural transformations as functors of $V$; the $\sym_m$-action on the right hand side is induced by the place permutation action on $V^{\otimes m}$.

\begin{exam}
\label{exam:schur_cat_lie}
For fixed $n \in \nat$, the Schur functor associated to $\cat \lie (-, n)$ (considered as a right $\rat\fb$-module) is the functor $V \mapsto \lie(V)^{\otimes n}$. This can be seen by using the fact that $\cat \lie (-,n)$ is isomorphic as a right $\rat \fb$-module to $\lie ^{\odot n}$ and the fact that the Schur construction is symmetric monoidal. 

 The above generalizes to treat the  right $\rat\fb$-module $t \mapsto \cat \lie (t, -)$ that takes values in left $\cat \lie$-modules. This has associated Schur functor 
\[
V \mapsto \underline{\lie (V)}, 
\]
considered as taking values in left $\cat \lie$-modules. This is natural with respect to $V\in\ob \fvs$. 
 
By the general identification (\ref{eqn:G_m}), for a specific $m$, the left $\cat \lie$-module $\cat \lie (m, -)$ is recovered  by 
\[
\cat \lie (m, -) 
\cong 
\nt (V^{\otimes m}, \underline{\lie (V)} ),
\] 
where the left $\cat \lie$-module structure is induced by that of $\underline{\lie (V)}$. 
\end{exam}

Theorem \ref{thm:catliemu} below adapts Example \ref{exam:schur_cat_lie} to the case of the set of projective generators (\ref{eqn:proj_gen_flie_mu}) of $\flie^\mu$. This uses the following:

\begin{lem}
\label{lem:fbop_mod_proj_gen_flie_mu}
The projective generators (\ref{eqn:proj_gen_flie_mu}) of $\flie^\mu$   form a  right $\rat\fb$-module taking values in left $\cat \lie$-modules, namely 
$
t \mapsto \cat \lie (t, -)^\mu.
$
\end{lem}

\begin{proof}
An immediate consequence of the fact that $t \mapsto \cat \lie (t, -)$ is a right $\rat\fb$-module taking values in left $\cat \lie$-modules and that $(-)^\mu$ is a functor.
\end{proof}

\begin{thm}
\label{thm:catliemu}
The Schur functor associated to $t \mapsto \cat \lie (t, -)^\mu$ is 
\[
V \mapsto H_0 (\lie(V); \underline{\lie (V)} ),
\]
 taking values in left $\cat \lie$-modules.

Hence, for $m \in \nat$, the projective cover $\cat \lie (m, -)^\mu$ of $\rat \sym_m$ in $\flie^\mu$ is isomorphic to the left $\cat \lie$-module:
\[
\nt (V^{\otimes m} , H_0 (\lie(V); \underline{\lie (V)} ) ).
\]
\end{thm}

\begin{proof}
The formation of the Schur functor (\ref{eqn:G_schur}) is an exact functor from right $\rat\fb$-modules to functors defined on finite-dimensional $\rat$-vector spaces. One deduces that this functor commutes with applying $(-)^\mu$ when considering right $\rat\fb$-modules with values in left $\cat \lie$-modules. 

Proposition \ref{prop:underline_g_mu} provides the natural  isomorphism 
\[
\underline{\lie (V)}^\mu \cong H_0 (\lie(V); \underline{\lie (V)} ).
\]
Combined with the identification given in Example \ref{exam:schur_cat_lie}, this gives the first statement. 

The second statement follows immediately from the general identification (\ref{eqn:G_m}) applied with $G$ the right $\rat\fb$-module of Lemma \ref{lem:fbop_mod_proj_gen_flie_mu}.
\end{proof}

\begin{rem}
Concentrating on homological degree zero gives only part of the story; the full structure underlying $H_* (\lie (V); \underline{\lie(V)})$ is explained in \cite{P_Hcatlie}.
\end{rem}

 \section{The proof of Theorem \ref{thm:outer_mu}}
 \label{sect:proof_lie_outer}

 This section gives the proof of Theorem \ref{thm:outer_mu}; the result is restated more explicitly below as Theorem \ref{thm:outer_mu_explicit} and then proved. The proof relies upon comparing the  definition of $\flie^\mu$ (Definition \ref{defn:fopd^mu}) with the characterization of $\foutanQ$ given by Proposition \ref{prop:fout_equivalent}.

%%%%%%%%%%%%%%%%%%%%%%%%%%%%%%%%%%%%%%%%%%%%%%%%%%%%%%%%%%%%%%%%%%%%%
\subsection{Identifying $\rho$}

The purpose of this section is to identify the natural transformation $\rbar : \tbar F \rightarrow F$ when $F$ is an analytic functor, using the equivalence of Theorem \ref{thm:grop_analytic}. 

\begin{lem}
\label{lem:identify_tau}
There is a natural equivalence of left $\rat\gr\op$, right $\cat\lie$ bimodules:
\[
\tau (\Phi (U \lie)) \cong \Phi (U\lie) \rconv U \lie,
\]
where $\rat\gr\op$ acts on the right hand side via its action on $\Phi (U \lie)$. 

Hence, for $M \in \ob \flie$, there is a natural isomorphism in $\f_\omega (\gr\op; \rat)$:
\[
\tau \big(\Phi (U \lie) \otimes_{\cat\lie} M \big) \cong  \big(\Phi (U\lie) \rconv U \lie \big) \otimes_{\cat\lie} M.
\]
\end{lem}

\begin{proof}
By construction of the exponential functor $\Phi (U \lie)$, one has for any $n \in \nat$
$$
\Phi (U \lie) (\zed^{\star n} \star \zed)=(U \lie) ^{\rconv n} \rconv U \lie
= \Phi (U\lie)(\zed^{\star n}) \rconv U \lie
,$$ 
where the final factor of $U \lie$ is constant (i.e., invariant under morphisms of $\gr\op$ of the form $f \star \id_\zed$).
From this, one deduces the first statement.
 
The second statement is an immediate consequence. 
\end{proof}

One can identify the natural transformation $\rho$ with respect to the identifications of Lemma \ref{lem:identify_tau}. For clarity, we first consider the case of an exponential functor $\Phi H$, where $H$ is a cocommutative Hopf algebra over $\rat$, using the material of \cite[Section 12]{PV}. For current purposes, $H$ can be taken to be primitively-generated; this ensures that $\Phi H$ is an analytic functor. Moreover, as in Lemma \ref{lem:identify_tau}, $\tau  \Phi H  \cong   \Phi H \otimes H$ in $\f(\gr\op; \rat)$. 

Now, $H$ is an $H$-bimodule for the left and right regular actions, hence can be considered as a right $H$-module with respect to the diagonal structure, i.e., for the right conjugation action $\ad : H \otimes H \rightarrow H$, 
\[
\ad \ : \ 
\tilde{h} \otimes h \mapsto \chi (h') \tilde{h} h''
\]
where $\Delta h = h' \otimes h''$ (using Sweedler notation) and $\chi$ is the Hopf algebra conjugation.

Thus, for each $n \in \nat$, one has the associated diagonal $H$-module structure  $\ad : (H^{\otimes n}) \otimes H \rightarrow H^{\otimes n}$. This induces a natural transformation:
\[
\ad_H : (\Phi H) \otimes H \rightarrow \Phi H
\]
in $\f (\gr\op; \rat)$ and, under the isomorphism, $\tau (\Phi H) \cong (\Phi H) \otimes H$, this identifies with $\rho_{\Phi H}$ (see \cite{PV}). 

\begin{rem}
This adjoint action is used by Conant and Kassabov in \cite{MR3546467}.
\end{rem}

The above considerations apply {\em mutatis mutandis} to  the cocommutative Hopf algebra $U \lie$ in $(\rmod, \rconv, \rat)$. This yields the natural transformation of left $\rat\gr\op$, right $\cat\lie$ bimodules
\[
\ad : (\Phi (U \lie)) \rconv U \lie \rightarrow \Phi (U \lie).
\]
This identifies with $\rho : \tau \big( \Phi (U \lie) \big) \rightarrow \Phi (U \lie)$ via the isomorphism of Lemma \ref{lem:identify_tau}.
 
\begin{rem}
Restricting the right $\cat\lie$-module structure to a right $\rat \fb$-structure, one can pass from $\ad$ to the morphism of associated Schur functors, which is a natural transformation of functors on $\gr\op$, natural with respect to $V \in \ob \fvs$: 
\[
\ad (V) \ : \  
\big(\Phi (U \lie) \rconv U \lie \big) (V) \rightarrow \big(\Phi (U \lie)\big) (V). 
\]
Now, $(U \lie)(V)$ is naturally isomorphic to $U (\lie (V))$, where $\lie(V)$ is the free Lie algebra on $V$; the Hopf algebra $U (\lie (V))$ is naturally isomorphic to the (primitively-generated) tensor Hopf algebra $T(V)$.   One deduces that $\big(\Phi (U \lie)\big) (V)$ is naturally isomorphic to $\Phi (U \lie(V)) \cong \Phi (T(V))$ as functors on $\gr\op$, naturally in $V$.

It follows that $\ad (V)$ identifies with 
\[
\ad_{T(V)} \ : \  \Phi (T(V)) \otimes T(V) \rightarrow \Phi (T(V)),
\]
defined as above, taking $H=T(V)$.
\end{rem}

Putting this together, one obtains:

\begin{prop}
\label{prop:identify_rho}
For $M \in \ob \flie$, the natural transformation $\rho : \tau \big(\Phi (U \lie) \otimes_{\cat\lie} M \big)
\rightarrow \Phi (U \lie) \otimes_{\cat\lie} M $ identifies, via the isomorphism of Lemma \ref{lem:identify_tau}, with 
\[
\ad \otimes \id_M \ : \  
\big( (\Phi (U \lie)) \rconv U \lie \big) \otimes_{\cat\lie} M 
\rightarrow 
 \Phi (U \lie) \otimes_{\cat\lie} M .
\]
\end{prop}

%%%%%%%%%%%%%%%%%%%%%%%%%%%%%%%%%%%%%%%%%%%%%%%%%%%%%%%%%%%%%%%%%%%%
\subsection{The reduced case}
Proposition \ref{prop:identify_rho} adapts to treat the reduced map $\rbar$ as follows.

\begin{nota}
\label{nota:ULie_bar}
Denote by $\overline{U \lie}$ the kernel of the augmentation $U \lie \rightarrow \rat$, where $\rat$ is considered as a right $\cat \lie$-module supported on $0$. 
\end{nota}

\begin{rem}
On passage to the associated Schur functors, $\overline{U \lie} (V)$ identifies as  the augmentation ideal of $U \lie (V) \cong T(V)$. 
\end{rem}

The following is clear:

\begin{lem}
\label{lem:split_ULie}
There is a canonical splitting 
$
U\lie \cong \overline{U\lie} \oplus \rat $ of right $\cat\lie$-modules. 
Moreover, the canonical inclusion $\lie \hookrightarrow U \lie$ maps to  $\overline{U \lie}$.
\end{lem}

One checks directly that the isomorphism of Lemma \ref{lem:identify_tau} restricts to 
\begin{eqnarray}
\label{eqn:tbar_Phi}
\tbar (\Phi (U \lie)) \cong \Phi (U\lie) \rconv \overline{U \lie}.
\end{eqnarray}
and hence that $\rbar : \tbar (\Phi (U \lie)) \rightarrow \Phi (U \lie)$ identifies with the following restriction of $\ad$:
\[
\overline{\ad} \ : \ (\Phi (U \lie)) \rconv \overline{U \lie} \rightarrow \Phi (U \lie).
\]

Then Proposition \ref{prop:identify_rho} gives:

\begin{cor}
\label{cor:identify_rbar}
For $M \in \ob \flie$, the natural transformation $\rbar : \tbar \big(\Phi (U \lie) \otimes_{\cat\lie} M \big)
\rightarrow \Phi (U \lie) \otimes_{\cat\lie} M $ identifies with 
\[
\overline{\ad} \otimes \id_M \ : \  
\big( (\Phi (U \lie)) \rconv \overline{U \lie} \big) \otimes_{\cat\lie} M 
\rightarrow 
 \Phi (U \lie) \otimes_{\cat\lie} M .
\]
\end{cor}

%%%%%%%%%%%%%%%%%%%%%%%%%%%%%%%%%%%%%%%%%%%%%%%%%%%%%%%%%%%%%%%%%%%%%%
\subsection{Comparing $\rbar$ and $\mut$}

By Proposition \ref{prop:mut_univ_example}, for $M \in \ob \flie$, the natural transformation $\mut : \delta M \rightarrow M$ identifies with 
\[
\mut \otimes \id_M \ : \ (\delta \cat\lie ) \otimes_{\cat\lie} M \rightarrow M,
\]
induced by the universal example $\mut : \delta \cat\lie \rightarrow \cat\lie$.

We require to understand the natural transformation that is obtained on applying the equivalence of categories $\Phi (U \lie) \otimes_{\cat\lie} - \ : \ \flie \rightarrow \f_\omega (\gr\op ; \rat)$. 

\begin{lem}
\label{lem:delta_cat_lie_grop}
There is a natural isomorphism of left $\rat\gr\op$, right $\cat\lie$ bimodules
\[
\Phi (U \lie) \otimes_{\cat\lie} \delta \cat\lie 
\cong 
\Phi (U \lie) \rconv \lie.
\]

Hence, for $M \in \ob \flie$, via the isomorphism $\delta M \cong (\delta \cat\lie) \otimes_{\cat\lie} M$ of Proposition \ref{prop:delta_universal}, there is a natural isomorphism:
\[
\Phi (U \lie) \otimes_{\cat\lie} \delta M
\cong 
(\Phi (U \lie) \rconv \lie)  \otimes_{\cat\lie} M.
\]
\end{lem}

\begin{proof}
The first statement follows from the isomorphism $\delta \cat\lie \cong \cat\lie \rconv \lie$ given by Proposition \ref{prop:delta_catopd_bimodule} by checking that $\Phi (U \lie) \otimes_{\cat\lie} (\cat\lie \rconv \lie)$ is isomorphic to $ 
\Phi (U \lie) \rconv \lie$ as a left $\rat\gr\op$, right $\cat\lie$ bimodule. This can be proved by using the description of $\delta \cat \lie$ given in Proposition \ref{prop:delta_proj} and then  checking the right $\cat\lie$-module structure.

One has the natural isomorphisms
\[
\Phi (U \lie) \otimes_{\cat\lie} \delta M
\cong 
\Phi (U \lie) \otimes_{\cat\lie} \big (\delta \cat\lie \otimes_{\cat\lie} M \big) 
\cong 
(\Phi (U \lie) \otimes_{\cat\lie} \delta \cat\lie)  \otimes_{\cat\lie} M, 
\]
where the first is given by Proposition \ref{prop:delta_universal} and the second by associativity of the tensor product. The second statement then follows from the first.
\end{proof}

The required compatibility result is the following:

\begin{prop}
\label{prop:compat_mut_ad}
The following diagram in left $\rat\gr\op$, right $\cat\lie$ bimodules commutes:
\[
\xymatrix{
\Phi (U \lie) \otimes_{\cat\lie} (\delta \cat\lie) 
\ar[rr]^(.6){\id \otimes \mut} 
\ar[d]_\cong 
&&
\Phi (U\lie) 
\ar@{=}[dd]
\\
\Phi (U \lie) \rconv \lie
\ar@{^(->}[d]
\\
\Phi (U \lie) \rconv \overline{U \lie}
\ar[rr]_{\overline{\ad}}
&&
\Phi (U \lie),
}
\]
in which the isomorphism is given by Lemma \ref{lem:delta_cat_lie_grop} and the inclusion is induced by the canonical inclusion $\lie \hookrightarrow \overline{U \lie}$ of right $\cat\lie$-modules.
\end{prop}

\begin{proof}
It suffices to show that, for each $n \in \nat$, the diagram commutes when evaluated on $\zed^{\star n}$. The diagram identifies as:
\[
\xymatrix{
(U \lie)^{\rconv n} \rconv \lie 
\ar[rr]
\ar@{^(->}[d]
&&
(U \lie)^{\rconv n}
\ar@{=}[d]
\\
(U \lie)^{\rconv n} \rconv \overline{U\lie} 
\ar[rr]_{\ad}
&&
(U \lie)^{\rconv n}
}
\]
in which the top map is induced by $\mut$. This is a diagram of right $\cat\lie$-modules.

Now, by construction $U \lie$ is generated under the associative product by $\lie$. This associative product is encoded in the structure of the exponential functor $\Phi (U \lie)$: explicitly,  applying $\Phi (U \lie)$ to the cogroup structure morphism $\zed \rightarrow \zed \star \zed$ of $\zed$ induces the product. Thus, by using the functoriality of $\Phi (U \lie)$ with respect to $\gr\op$, one can show that it suffices to establish that the diagram commutes when restricted to $\lie^{\rconv n} \rconv \lie \subset (U \lie)^{\rconv n} \rconv \lie$, where the inclusion is induced by $\lie \subset U \lie$. 

This reduces to showing the commutativity of the following diagram
\[
\xymatrix{
\lie^{\rconv n} \rconv \lie 
\ar[rr]^{\mut}
\ar@{^(->}[d]
&&
\lie^{\rconv n}
\ar@{^(->}[d]
\\
(U \lie)^{\rconv n} \rconv \overline{U\lie}
\ar[rr]_{\overline{\ad}}
&&
(U \lie)^{\rconv n}.
}
\]
The commutativity follows from  the relationship between the adjoint representation of $\lie$ and the conjugation action of $U \lie$ on itself, together with its extension to the `tensor' products. Indeed, it suffices to check the commutativity of the diagram after passage to the associated Schur functors. The diagram then becomes
\[
\xymatrix{
\lie(V) ^{\otimes n} \otimes  \lie(V) 
\ar[rr]
\ar@{^(->}[d]
&&
\lie(V)^{\otimes n}
\ar@{^(->}[d]
\\
(U \lie(V))^{\otimes n} \otimes \overline{U\lie}(V) 
\ar[rr]_{\overline{\ad}(V)}
&&
(U \lie (V))^{\otimes n},
}
\]
where the top horizontal map is the action of $\lie (V)$ on $\lie(V)^{\otimes n}$ given by the tensor product of the adjoint representation. Commutativity is checked directly. 
\end{proof}

\begin{cor}
\label{cor:mut_rho}
For $M \in \ob \flie$, the following diagram in $\f_\omega (\gr\op; \rat)$ commutes:
\[
\xymatrix{
\Phi (U\lie) \otimes_{\cat\lie} \delta M 
\ar[rr]^{\id \otimes \mut}
\ar[d]_\cong 
&&
\Phi (U\lie) \otimes_{\cat\lie} M
\ar@{=}[dd]
\\ 
(\Phi (U \lie) \rconv \lie)  \otimes_{\cat\lie} M
\ar@{^(->}[d]
\\
(\Phi (U \lie) \rconv \overline{U \lie})  \otimes_{\cat\lie} M
\ar[rr]_{\overline{\ad} \otimes \id}
&&
\Phi (U\lie) \otimes_{\cat\lie} M,
}
\]
in which the isomorphism is given by Lemma \ref{lem:delta_cat_lie_grop} and the bottom left hand vertical map is injective. 
\end{cor}

\begin{proof}
The commutativity of the diagram follows from that of the `universal example', given by Proposition \ref{prop:compat_mut_ad}, in conjunction with the identification of $\mut$ given by Proposition \ref{prop:mut_univ_example} and using the identification of Lemma \ref{lem:delta_cat_lie_grop} for the top left corner.

The injectivity of the indicated map is a consequence of the fact that the inclusion $\lie \hookrightarrow \overline{U \lie}$ admits a retract in right $\cat\lie$-modules, which follows from the Poincaré-Birkhoff-Witt theorem, interpreted at the level of operads.
\end{proof}

\begin{rem}
\label{rem:identify_rbar}
\ 
\begin{enumerate}
\item 
The bottom horizontal map of the diagram in the statement of Corollary \ref{cor:mut_rho} identifies with $\rbar : \tbar \big(\Phi (U\lie) \otimes_{\cat\lie} M\big) \rightarrow \Phi (U\lie) \otimes_{\cat\lie} M$ using the natural isomorphism 
\[
\tbar \big(\Phi (U\lie) \otimes_{\cat\lie} M\big)
\cong 
\big( \Phi (U\lie) \rconv \overline{U \lie}\big) \otimes_{\cat\lie} M
\]
induced by (\ref{eqn:tbar_Phi}), together with Corollary \ref{cor:identify_rbar}.
\item 
The injectivity statement is not required below. It is included since it is of independent interest.
\end{enumerate}
\end{rem}

%%%%%%%%%%%%%%%%%%%%%%%%%%%%%%%%%%%%%%%%%%%%%%%%%%%%%%%%%%%%%%%%%%%%%%%%
\subsection{The proof}

Theorem \ref{thm:outer_mu} follows from the more explicit result:

\begin{thm}
\label{thm:outer_mu_explicit}
For $M \in \ob \flie$, the following are equivalent:
\begin{enumerate}
\item 
$M$ belongs to $\flie^\mu$; 
\item 
the analytic functor $\Phi (U \lie) \otimes _{\cat\lie} M$ belongs to $\foutanQ$. 
\end{enumerate}
\end{thm}

\begin{proof}
By definition the two conditions in the  statement are equivalent respectively to:
\begin{enumerate}
\item 
$\mut : \delta M \rightarrow M$ is zero; 
\item 
$\rbar : \tbar \big (\Phi (U \lie) \otimes _{\cat\lie} M \big) \rightarrow \Phi (U \lie) \otimes _{\cat\lie} M $ is zero.
\end{enumerate}
These conditions are compared by using the results of the preceding subsections.

Using the identification observed in Remark \ref{rem:identify_rbar}, the outer commutative square of Corollary \ref{cor:mut_rho} can be written as:
\[
\xymatrix{
\Phi (U\lie) \otimes_{\cat\lie} \delta M 
\ar[rr]^{\id \otimes \mut}
\ar[d]
&&
\Phi (U\lie) \otimes_{\cat\lie} M
\ar@{=}[d]
\\
\tbar \big (\Phi (U \lie) \otimes _{\cat\lie} M \big)
\ar[rr]_\rbar
&&
\Phi (U \lie) \otimes _{\cat\lie} M.
}
\]

Under the equivalence of Theorem \ref{thm:grop_analytic}, the top horizontal map is zero if and only if $\mut$ is.
 It follows  that, if $\rbar$ is zero, then so is $\mut$. 

It remains to prove the converse. For this we work with the (equivalent) outer square as in Corollary \ref{cor:mut_rho}:
\[
\xymatrix{
\Phi (U\lie) \otimes_{\cat\lie} \delta M 
\ar[rr]^{\id \otimes \mut}
\ar[d]
&&
\Phi (U\lie) \otimes_{\cat\lie} M
\ar@{=}[d]
\\
(\Phi (U \lie) \rconv \overline{U \lie})  \otimes_{\cat\lie} M
\ar[rr]_\rbar
&&
\Phi (U \lie) \otimes _{\cat\lie} M.
}
\]

Now, the argument used to deduce Corollary \ref{cor:mut_rho} from Proposition \ref{prop:compat_mut_ad} shows that
the above commutative diagram is obtained by applying $-\otimes_{\cat\lie} M $ to the commutative diagram
\[
\xymatrix{
(\Phi U \lie) \rconv \lie 
\ar[rr]^{\ad|_\lie}
\ar[d]
&&
\Phi U \lie 
\ar@{=}[d]
\\
(\Phi U \lie) \rconv \overline{U  \lie} 
\ar[rr]_{\overline{\ad}}
&&
\Phi U \lie ,
}
\]
where the left hand vertical arrow is induced by the inclusion $\lie \hookrightarrow \overline{U \lie}$. 

The key observation is that the images of the horizontal maps are the same:
\[
\mathrm{Image} (\ad|_\lie) = 
\mathrm{Image} (\overline{\ad})
\]
as sub right $\cat\lie$-modules of $\Phi U \lie$. This follows  from the fact that $\lie$ generates $U \lie$ as a unital, associative algebra in $(\rmod, \rconv , \rat)$. 

Hence the above commutative diagram gives the commutative diagram:
\[
\xymatrix{
(\Phi U \lie) \rconv \lie 
\ar@{->>}[r]^(.6){\ad|_\lie}
\ar[d]
&
\mathrm{Image}
\ar@{^(->}[r]
\ar@{=}[d]
&
\Phi U \lie 
\ar@{=}[d]
\\
(\Phi U \lie) \rconv \overline{U  \lie} 
\ar@{->>}[r]_(.6){\overline{\ad}}
&
\mathrm{Image}
\ar@{^(->}[r]
&
\Phi U \lie, 
}
\]
writing $\mathrm{Image}$ for the common image.

On applying $-\otimes_{\cat\lie} M$ one obtains that the vanishing of the top row is equivalent to that of the bottom row, since both these conditions are equivalent to the map 
\[
\mathrm{Image} \otimes_{\cat\lie} M 
\rightarrow 
(\Phi U \lie) \otimes_{\cat\lie} M
\]
being zero, using the fact that $-\otimes_{\cat\lie} M$ is right exact. This concludes the proof.
\end{proof}

%\nocite{*}
%\bibliographystyle{amsalpha}
%\bibliography{out.bib}
\providecommand{\bysame}{\leavevmode\hbox to3em{\hrulefill}\thinspace}
\providecommand{\MR}{\relax\ifhmode\unskip\space\fi MR }
\providecommand{\MRhref}[2]{%
  \href{http://www.ams.org/mathscinet-getitem?mr=#1}{#2}
}
\providecommand{\href}[2]{#2}

\end{document}